\documentclass[leqno]{amsart}
\usepackage{amsmath}
\usepackage{amssymb}
\usepackage{hyperref}

\newcommand\bQ{\ensuremath{\mathbb{Q}}}
\newcommand\bF{\ensuremath{\mathbb{F}}}
\newcommand\bP{\ensuremath{\mathbb{P}}}
\newcommand\bZ{\ensuremath{\mathbb{Z}}}
\newcommand\bC{\ensuremath{\mathbb{C}}}

\newcommand\cL{\ensuremath{\mathcal{L}}}
\newcommand\cO{\ensuremath{\mathcal{O}}}
\newcommand\cP{\ensuremath{\mathcal{P}}}
\newcommand\cS{\ensuremath{\mathcal{S}}}
\newcommand\cT{\ensuremath{\mathcal{T}}}
\newcommand\sD{\ensuremath{\mathsf{D}}}
\newcommand\mysf{\ensuremath{\mathsf{f}}}
\newcommand\sF{\ensuremath{\mathsf{F}}}
\newcommand\sM{\ensuremath{\mathsf{M}}}
\newcommand\sP{\ensuremath{\mathsf{P}}}
\newcommand\sQ{\ensuremath{\mathsf{Q}}}
\newcommand\sS{\ensuremath{\mathsf{S}}}
\newcommand\sr{\ensuremath{\mathsf{r}}}
\newcommand\myss{\ensuremath{\mathsf{s}}}
\newcommand\st{\ensuremath{\mathsf{t}}}
\newcommand\su{\ensuremath{\mathsf{u}}}
\newcommand\sx{\ensuremath{\mathsf{x}}}
\newcommand\sy{\ensuremath{\mathsf{y}}}
\newcommand\sz{\ensuremath{\mathsf{z}}}
\newcommand\ve{\ensuremath{\varepsilon}}
\newcommand\phr{\ensuremath{\phantom{r}}}

\DeclareMathOperator\tr{tr}
\DeclareMathOperator\Jac{Jac}
\newcommand\case[1]{{\ensuremath{{\mathrm(}{\mathit{#1\/}}{\mathrm)}}}}

\theoremstyle{plain}
\newtheorem{theorem}{Theorem}[section]
\newtheorem{lemma}[theorem]{Lemma}
\newtheorem{prop}[theorem]{Proposition}
\newtheorem{cor}[theorem]{Corollary}
\newtheorem*{propnn}{Proposition~\ref{prop:divCantor}}

\theoremstyle{remark}
\newtheorem{rem:helltorsionbound}[theorem]{Remarks}
\newtheorem{rem:Cantor}[theorem]{Remark}
\newtheorem{rem:endintro}[theorem]{Remarks}
\newtheorem{rem:genbound}[theorem]{Remark}
\newtheorem{rem:cexhyperellbound}[theorem]{Remark}
\newtheorem{rem:deltabound}[theorem]{Remark}
\newtheorem{rem:gammabound}[theorem]{Remark}
\newtheorem{defn}[theorem]{Definition}
\newtheorem{defnPiDelta}[theorem]{Definition}
\newtheorem{ex:cantor}[theorem]{Example}
\newtheorem{ex:sixtorsion}[theorem]{Example}
\newtheorem{ex:sixtorsionfive}[theorem]{Example}
\newtheorem{ex:seventorsion}[theorem]{Example}
\newtheorem{ex:sevengthree}[theorem]{Example}
\newtheorem{ex:x015w3}[theorem]{Example}

\title[Number of points of given order]{On the number of points of given order on odd-degree hyperelliptic curves}
\author{John Boxall}
\address{Laboratoire de Math\'ematiques Nicolas Oresme, UMR CNRS 6139, Campus 2, Universit\'e de Caen-Normandie, 14032 Caen cedex, France}
\email{john.boxall@unicaen.fr}
\thanks{I would like to thank David Grant for drawing my attention to the work of Cantor \cite{Ca94} and \^Onishi \cite{On05} and his numerous comments on preliminary versions of this paper.  I would also like to thank the referees for suggesting a number of improvements and a comment that led to the addition of Remark~\ref{rem:Cantor}.}
\keywords{Algebraic curves, Jacobians, Torsion points}

\date{\today}

\keywords{Hyperelliptic curves, Jacobian varieties, Torsion points}
\subjclass{14H40, 14H45, 14G17}

\begin{document}

\begin{abstract}
For integers $N\geq 2$ and $g\geq 1$, we study bounds on the cardinality of the set of points of order dividing $N$ lying on a hyperelliptic curve of genus $g$ embedded in its jacobian using a Weierstrass point as base point. This leads us to revisit division polynomials introduced by Cantor in 1995 and strengthen a divisibility result proved by him. Several examples are discussed. 
\end{abstract}

\maketitle

\section{Introduction}
Let $X$ be an irreducible smooth proper curve of genus $g\geq 2$ over an algebraically closed field $k$ and let $J$ be the jacobian variety of $X$. Fix a closed point $\infty\in X$ and embed $X$ in $J$ using $\infty$ as base point. Thus $\xi\in X$ is mapped to the divisor class of $[\xi]-[\infty]$. If $N$ is a positive integer, we denote by $J[N]$ the subgroup of $J(k)$ of points of order dividing $N$, by $J[N]^*$ the subset of points of exact order $N$ and by $\tilde{J}[N]$ the set $J[N]-J[2]$. Throughout the paper, we reserve the letter $p$ for the characteristic of $k$, allowing $p=0$. If $p>0$, we say that $p$ is \emph{purely inseparable for $X$} if the multiplication-by-$p$ isogeny of $J$ is purely inseparable. One has the following bound on $\sharp\, (X\cap J[N])$. 

\begin{prop} \label{prop: generalbound}  Keep the notation just introduced.  

\case{a} $X\cap J[2]^*$ is empty unless $X$ is hyperelliptic and $\infty$ is a Weierstrass point, in which case 
$\sharp\, (X\cap J[2]^*) = 2g+1$ if $p\neq 2$ and $\sharp\, (X\cap J[2]^*) \leq g$ if $p=2$.
                                                         
\case{b} Suppose that $N\geq 3$. We have
\begin{equation*}
\sharp\, (X\cap J[N])\leq g(N-1)^2,
\end{equation*}
unless $p$ is purely inseparable for $X$ and $N-1$ is a power of $p$, in which case
\begin{equation*}
\sharp\, (X\cap J[N])\leq \begin{cases} 25g& \text{ if $N=3$ and $p=2$,}\\
                                                               g(N+1)^2 & \text{otherwise. }
                                                                               \end{cases}
\end{equation*}                                                                          
\end{prop}  

Although this is an easy consequence of the Riemann-Roch theorem and well-known properties of the jacobian, we have been unable to find this or a similar statement in the literature. The proof is a simple adaptation to arbitrary characteristic of an idea of Coleman-Kaskel-Ribet \cite{CoKaRi99}. Since it uses methods of a somewhat different nature to those of the rest of the paper, we relegate it to the Appendix (see \S~\ref{sec:proofgeneralbound}). 

One can ask whether these bounds are sharp. The main purpose of this paper is to show that, at least in the case of hyperelliptic curves and $\infty$ is a Weierstrass point, very often they are not.  

So suppose from now on that $X$ be a hyperelliptic curve and that $\infty$ is a Weierstrass point.  Then $X-\{\infty\}$ is affine and has a model of the form $y^2+Q(x)y=P(x)$ where $P(x)$ is a monic polynomial of degree $2g+1$ and $Q(x)$ has degree at most $g$. (When $p\neq 2$, such a model exists with $Q(x)=0$.) We denote by $X_{P,Q}$ the curve corresponding to the polynomials $P$ and $Q$ and by $J_{P,Q}$ its jacobian. When there is no risk of confusion we usually abbreviate these to $X$ and $J$ and similarly for the notations $\tilde{\cT}_{P,Q,N}$, $\tilde{U}_{P,Q,N}(x)$, \dots introduced below.  However we shall always keep the index $N$ (which denotes an integer not less than $2$). We always think of $X$ as embedded in $J$ using $\infty$ as base point. 

The case $N=2$ is covered by Proposition~\ref{prop: generalbound}.  When $3\leq N\leq 2g$ (and $p\neq 2$), Zarhin \cite{Za19} has proved that $X\cap \tilde{J}[N]$ is empty (see also Lemma~\ref{lem:first}). Here we give bounds for $\sharp (X\cap \tilde{J}[N])$ when $N\geq 2g+1$.  

The hyperelliptic involution $\iota$ of $X$ fixes $\infty$ and sends the point $(x_0,y_0)$ to $(x_0, -Q(x_0)-y_0)$. Also, $(x_0,y_0)$ has order $N$ if and only if $\iota(x_0,y_0)$ has order $N$. Hence, since $N\geq  3$, points of order $N$ come in pairs with the same $x$-coordinate $x_0$. Define
\begin{equation*}
\tilde{\cT}_N=\tilde{\cT}_{P,Q,N}=\{x_0\in k\mid (x_0,y_0) \text{ and } (x_0,-Q(x_0)-y_0)  \in \tilde{J}[N]\}
\end{equation*}
and
\begin{equation*}
\tilde{U}_N(x)=\tilde{U}_{P,Q,N}(x)=\prod_{x_0\in \tilde{\cT}_{N}}(x-x_0).
\end{equation*} 

Let $\{p_i\}_{0\leq i\leq 2g}$ and $\{q_i\}_{0\leq i\leq g}$ be two sets of variables and let $\Lambda=\bZ[p_0,\dots, p_{2g},q_0,\dots q_g]$ be the polynomial ring. Set $\sP(\sx)=p_0+p_1\sx+\cdots +p_{2g}\sx^{2g}+\sx^{2g+1}$ and $\sQ(\sx)=q_0+q_1\sx+\cdots +q_{g}\sx^g$, so the {\lq\lq}universal{\rq\rq} hyperelliptic curve of genus $g$ has equation
\begin{equation*}
\sy^2+\sQ(\sx)\sy=\sP(\sx)
\end{equation*}
and the equation for $X_{P,Q}$ is obtained by specializing the coefficients of $\sP$ and $\sQ$ to those of $P$ and $Q$, $\sx$ to $x$ and $\sy$ to $y$. For every integer $N\geq 2g+1$, we construct a {\lq\lq}universal{\rq\rq} polynomial $\Delta_{\sP,\sQ,N}(\sx)\in \Lambda[\sx]$ (see Definition~\ref{defPiDelta}). Denote by $\varDelta_{P,Q,N}(x)$ the specialization of $\Delta_{\sP,\sQ,N}(\sx)$ to the curve $X_{P,Q}$ and abbreviate $\varDelta_{P,Q,N}(x)$ to $\varDelta_N(x)$ if there is no risk of confusion. Thus $\varDelta_N(x)$ is an element of $k[x]$. Then $\Delta_{\sP,\sQ,N}(\sx)$ and $\varDelta_{N}(x)$ enjoy the following properties. 

\begin{theorem} \label{th:helltorsionbound} Suppose that $N\geq 2g+1$, and that $X=X_{P,Q}$ is a hyperelliptic curve of genus $g$. 

\case{a} Define 
\begin{equation*}
\delta_{N,g}=\begin{cases} g(N+2)(N-2g)/2  \text{ if $N\geq 2g+2$ is even},\\ 
                                                            g(N+1)(N-2g+1)/2  \text{ if $N\geq 2g+1$ is odd}.\end{cases}
\end{equation*}
Then $\Delta_{\sP,\sQ,N}(\sx)$ has degree $\delta_{N,g}$ as a polynomial in $\sx$. Hence $\varDelta_N(x)$ has degree at most $\delta_{N,g}$. Furthermore, $\tilde{U}_N(x)^g$ divides $\varDelta_N(x)$.  

\case{b} In particular, if $\varDelta_N(x)\neq 0$, then the $\sharp \, (X\cap \tilde{J}[N])\leq 2\frac{\deg{(\varDelta_N(x))}}{g}$. 

\case{c} Suppose that $p=0$ or $p> N$. Then $\varDelta_N(x)$ has degree equal to $\delta_{N,g}$. In particular, $\varDelta_N(x)\neq 0$ and so 
\begin{equation*}
\sharp\, (X\cap \tilde{J}[N])\leq \begin{cases} (N+2)(N-2g)  \text{ if $N\geq 2g+2$ is even},\\ 
                                                            (N+1)(N-2g+1)  \text{ if $N\geq 2g+1$ is odd}.\end{cases}
\end{equation*}\end{theorem}

When $2\leq p\leq N$, the degree of $\Delta_N(x)$ can be strictly less than $\delta_{N,g}$. Thus, if $\varDelta_N(x)\neq 0$, the bound on $\sharp \, (X\cap \tilde{J}[N])\leq 2\frac{\deg{(\varDelta_N(x))}}{g}$ in \case{b} may be better than the one given in \case{c}. Hence in all cases where $\varDelta_N(x)\neq 0$, Theorem~\ref{th:helltorsionbound} is a considerable improvement on Proposition~\ref{prop: generalbound},  since it replaces the $O(gN^2)$ bound there by a $O(N^2)$ bound.

However it is important to stress that the condition $\varDelta_N(x)\neq 0$ is essential. For example, let $X$ be the curve with a affine model $y^2+y=x^5$, and take $p=2$ and $N=5$. We shall see in Example~\ref{ex:cantor} that $\varDelta_5(x)=0$ and that $\sharp\, (X\cap \tilde{J}[5])=32$, whereas, in cases where it applies Theorem~\ref{th:helltorsionbound} gives $\sharp\, (X\cap \tilde{J}[5])\leq 12$. Similarly when $p=5$ and $X$ is the curve $y^2=x^5-x$ (see Example~\ref{ex:sixtorsion}), then $\varDelta_{6}(x)$ vanishes and $\sharp\, (X \cap \tilde{J}[6])=40$ whereas  Theorem~\ref{th:helltorsionbound}  would give $\sharp\, (X\cap \tilde{J}[6])\leq 16$.  Indeed, in general we do not know how to improve on Proposition~\ref{prop: generalbound} when $2\leq p\leq N$. Nevertheless, we can do this when $N$ is not much bigger than $2g+1$. For example, we have:

\begin{prop} \label{prop:2g12} For all hyperelliptic curves $X$ of genus $g$ in any characteristic, we have
\begin{equation*}
\sharp\, (X\cap \tilde{J}[2g+1])\leq 8g^2   \quad \text{and} \quad  \sharp\, (X\cap \tilde{J}[2g+2])\leq 8g^2+4g.
\end{equation*}
\end{prop}

This is better than Proposition~\ref{prop: generalbound}, which gives a bound that is cubic in $g$. But it is weaker than Theorem~\ref{th:helltorsionbound} when it applies, which gives bounds that are linear in $g$.  The examples mentioned above show that the bound for $\sharp\, (X\cap \tilde{J}[2g+1])$ is attained by the curve $y^2+y=x^5$ when $p=2$ and $N=5$ and that for $\sharp\, (X\cap \tilde{J}[2g+2])$ is attained by the curve $y^2=x^5-x$ when $p=5$ and $N=6$.  See  Proposition~\ref{prop:badgenbound} for the general result when $2g+1$ and $2g+2$ are replaced by an arbitrary integer $N\geq 2g+1$.

\begin{rem:helltorsionbound}
\case{a} The example of $y^2+y=x^5$ when $p=2$ and $N=5$ also shows that the general bound $g(N-1)^2$ of Proposition~\ref{prop: generalbound}~\case{b} can fail, since it predicts $\sharp\, (X\cap J[5])\leq 32$. But since $J[5]=\tilde{J}[5]\cup \{\infty\}$,  $\sharp\, (X\cap J[5])= 33$.    This shows that it is necessary to take some account of inseparability in Proposition~\ref{prop: generalbound}. 

\case{b} When $g=2$, $Q(x)=0$  and $p\neq 2$, $\varDelta_5(x)$ is equal to the polynomial appearing in Proposition~1.2 of \cite{BoGrLe01} (with $P(x)$ equal to the polynomial written $f(x)$ there).  That proposition asserted that if $f(x_0)\neq 0$, then $x_0\in \tilde{\cT}_5$ if and only if $(x-x_0)^2$ divides $\varDelta_5(x)$. In fact, Theorem~\ref{th:helltorsionbound} may be viewed as a generalization of the {\lq\lq}only if{\rq\rq} part of that proposition. The examples above with vanishing $\varDelta_N(x)$ show that we cannot generalize the {\lq\lq}if{\rq\rq} part in such a simple way.  To do this, and to give a method of determining the sets $\tilde{\cT}_N$ in general, we shall need to construct other polynomials; $\tilde{U}_N(x)$ will then be (up to taking the radical and removing possible factors corresponding to points of order $2$) equal to the gcd of these polynomials. 

\case{c} Our proof of Theorem~\ref{th:helltorsionbound} also works when $g=1$, in which case $X=J$ is an elliptic curve with origin $\infty$. Thus $X\cap \tilde{J}[N]=\tilde{J}[N]$ and our bound is actually an equality when $N$ is prime to $p$. When $P(x)=x^3+ax+b$ and $Q(x)=0$, $\varDelta_N(x)$ is equal, up to sign, to the classical $N^{\text{th}}$ division polynomial attached to an elliptic curve (see for example \cite{Lang78}, Chapter~2). 
\end{rem:helltorsionbound} 

Here is a brief sketch of the strategy of proof of Theorem~\ref{th:helltorsionbound} and Proposition~\ref{prop:2g12}. We refer to \S~\ref{sec:hyperelliptic} for details. Write $\mu=\lfloor{\frac{N-2g-1}{2}}\rfloor$. To the hyperelliptic curve $X=X_{P,Q}$ above, we associate the polynomial $F(x)=F_{P,Q}(x)=4P(x)+Q(x)^2$. The hypothesis that $X$ is smooth implies that $F(x)\neq 0$.  Applying the Riemann-Roch theorem to $X$, we construct a polynomial matrix $M_{N}(x)=M_{P,Q,N}(x)$ with $\mu+1$ rows and $\mu+g$ columns and prove that it has maximal rank, and that if $x_0\in k$ satisfies $F(x_0)\neq 0$, then a necessary and sufficient condition for $x_0$ to belong to $\tilde{\cT}_{N}$ is that  the rank of $M_{N}(x_0)$ is not maximal. Thus $\tilde{U}_{P,Q,N}(x)$ is equal (possibly after removing factors common with $F(x)$ and taking the radical) to the gcd of the $(\mu+1)\times (\mu+1)$ subdeterminants of $M_{N}(x)$. In fact we shall see that all these subdeterminants are divisible by $F(x)^{\mu(\mu+1)/2}$ (see Proposition~\ref{prop:powfxdivSigma}). We prove Proposition~\ref{prop:2g12}, and its generalizations to all $N\geq 2g+1$ (see Theorem~\ref{th:genbound} and the results following it) by bounding the degrees of these subdeterminants. We use Hasse-Schmidt derivatives in order to obtain results in arbitrary characteristic. 

In fact, the matrix $M_N(x)$ and many of the other objects  that we define are {\lq\lq}universal{\rq\rq}, in the sense that they are defined for a universal hyperelliptic curve $\sy^2+\sQ(\sx)\sy=\sP(\sx)$ of genus $g$, where the coefficients  of $\sQ(\sx)$ and $\sP(\sx)$ are indeterminates.  Thus \S~\ref{sec:HSanduniv} is devoted to the study of these universal objects.

Of particular importance is the {\lq\lq}leftmost{\rq\rq} subdeterminant of $M_{P,Q,N}(x)$ (see just after Lemma~\ref{lem:leadcoeffmat} and the beginning of \S~\ref{sec:hyperelliptic}). We denote this subdeterminant  by $\varGamma_{P,Q,N}(x)$ or simply $\varGamma_N(x)$. Then $\varDelta_{P,Q,N}(x)$ is defined to be the quotient of $\varGamma_{N}(x)$ by $F(x)^{\mu(\mu+1)/2}$.  We prove (see Theorem~\ref{th:genbound}) that all the subdeterminants of $M_N(x)$ are divisible by certain powers of $\tilde{U}_N(x)$ and in particular that $\tilde{U}_N(x)^g$ divides $\Delta_N(x)$ (see \ref{subsec:bestworst}.) 

When the paper was nearing completion, the author realized that $\varDelta_N$ is equal, up to sign, to the polynomial denoted $P_{N-g+1}$ in Cantor's paper  \cite{Ca94}.  (See also \^Onishi~\cite{On05} for related work when $k=\bC$.)  This is independent of the results of the present paper, but we have included a proof in \ref{subsec:Cantor}.  Cantor proved that $x_0\in \tilde{\cT}_{N}$ if and only if $P_{N-g+1+r}(x_0)=0$ for all $r\in \{0,1,\dots, 2g-2\}$. We improve on this by proving that certain \emph{powers} of $\tilde{U}_{N}(x)$ divide these $P_{N-g+1+r}(x)$'s. Returning to our notation, the precise result is as follows. 

\begin{prop} \label{prop:divCantor} Let $N\geq 2g+1$ and let $r\in \{0,1,\dots ,2g-2\}$. Define 
\begin{equation*}
\ve_{r,g}=\left(g-\left\lfloor{\frac{r+1}{2}}\right\rfloor \right)\left(\left\lfloor{\frac{r}{2}}\right\rfloor+1    \right).
\end{equation*}
Then $\tilde{U}_{N}(x)^{\ve_{r,g}}$ divides $\varDelta_{N+r}(x)$. 
\end{prop}

Note the symmetry $\ve_{2g-2-r,g}=\ve_{r,g}$. It is easy to see that $\ve_{0,g}=g$, $\ve_{1,g}=g-1$ and $\ve_{r,g}\geq g-1$ for all $g$ and $r$. When $g\in \{1,2,3,4\}$ the values of $\ve_{r,g}$ are as follows:
\begin{equation*}
\begin{array}{c||c|c|c|c|c|c|c}
g\ r  &     0 & 1 & 2 &  3 & 4 & 5 & 6\\
\hline
1\  \phr&1 & {} & {} & {} & {} & {} & {}\\
2\  \phr&    2 & 1 & 2 &  {} & {} & {} & {}\\
3\  \phr&    3 & 2 & 4 & 2 &  3  & {} & {}\\
4\  \phr&    4 & 3 & 6 & 4 & 6   & 3 & 4 
\end{array}.
\end{equation*}

The proof will be given in \S~\ref{sec:Cantorprop}.  Note that, by the result of Zarhin mentioned above, whenever $\ell$ runs through a sequence of consecutive integers of length at most $2g-1$, the polynomials $\tilde{U}_{\ell}(x)$ are pairwise coprime. Thus, if $N\geq 2g+1$, the Proposition implies that $\varDelta_{N}(x)$ is divisible by a product of positive powers of $\tilde{U}_{\ell}(x)$ as $\ell$ runs over the integers between $N-2g+2$ and $N$ such that $\ell\geq 2g+1$.  When $\Delta_N(x)$ does not vanish, we can interpret this as a obtain a bound on a weighted sum of the cardinalities of the sets $X\cap \tilde{J}[\ell]$. 

We shall prove that, when $p=0$ or $p>N$, then $\varDelta_{N}(x)$ has degree $\delta_{N,g}$. In particular, it does not vanish.  As already noted, $\varDelta_{N}(x)$ \emph{can} vanish when $0<p\leq N$.  

\begin{rem:Cantor} \label{rem:Cantor}
Denote by $\Theta$ the theta divisor associated to $X$, i.~e. $\Theta=\{\xi_1+\xi_2+\cdots +\xi_{g-1}\mid (\xi_1,\xi_2,\dots, \xi_{g-1})\in X^{g-1}\}$ (recall that we are viewing $X$ as embedded in $J$ using $\infty$ as base point).  Cantor proved that the zeros of $P_r$ are the $x$-coordinates of the points $(x,y)\in X-\{\infty\}$ belonging to the support of the zero-cycle $X.r^*\Theta$ (counted with multiplicity).  This leads to a geometric interpretation of the degree of $P_r$: it is the equal to the degree of the zero-cycle $X_0.r^*\Theta$, where $X_0$ is the affine curve $X-\{\infty\}$. (It is well-known that the degree of $X.r^*\Theta$ is equal to $gr^2$; this follows for example from the computation at the end of the proof of Lemma~\ref{lem:CKRlemma}.)    

Since $X$ is irreducible of dimension one,  it follows that the vanishing of $P_r$ (or of $\Delta_{r+g-1}$) is equivalent to $r_*X\subseteq \Theta$. In our exemples where $\Delta_N$ vanishes (Exemples~\ref{ex:cantor} and \ref{ex:sixtorsion}), this inclusion can of course be checked directly.

Cantor also proved recurrence relations between the $P_r$, generalizing those when $g=1$. These do not seem to be easily accessible by our methods, which depend on the fact that $\xi\in X-\{\infty\}$ has order dividing $N$ if and only if there is a function on $X$ with divisor $N([\xi]-[\infty])$.
\end{rem:Cantor}

It is easy to compute $\tilde{U}_{N}(x)$ in simple cases and we provide several examples illustrating this in \S~\ref{sec:examples}.  Indeed, it was these and similar computations that led to the discovery of Proposition~\ref{prop:divCantor}. In that section, we also describe a family of curves $X$ for which there are infinitely many integers $N$ such that $\sharp\, (X\cap J[N])\geq N^2$.  Note that, by the Manin-Mumford conjecture (first proved by Raynaud \cite{Ra83}), this cannot happen in characteristic zero.  Results analogous to the Manin-Mumford conjecture have been proved in positive characteristic when $X$ does not come from a curve over a finite field by extension of scalars (see for example \cite{PiRo04}). In our examples, $k$ is an algebraic closure of a finite field and the curve $y^2+y=x^5$ already mentioned is a member of this family when $p=2$, or indeed whenever $p\equiv 2$ or $3\pmod{5}$. 

\begin{rem:endintro} 
(1) In \cite{BeZa19}, Bekker and Zarhin have made a detailed study of the case $N=2g+1$ for any $g$. As an example of their results, they prove that if $2g+1$ is a power of the characteristic of $k$, then $X\cap J[2g+1]^*$ contains at most two points. When it applies, this result is stronger than Theorem~\ref{th:helltorsionbound}. More generally, the point of view of \cite{BeZa19} is more arithmetic than ours.  In this paper we concentrate on bounds over algebraically closed fields. 

(2) In \cite{Par21}, Pareschi proves that when $g=2$ and $k=\bC$, we have $\sharp (X\cap J[N])\leq \frac{3}{2}N^2$ whatever the base point used for the embedding of $X$ in $J$ (see the section entitled \emph{Proof for $g=2$} at the end of \cite{Par21}).  This is stronger than Proposition~\ref{prop: generalbound}. But when the  base point is a Weierstrass point, it is weaker than Theorem~\ref{th:helltorsionbound}.
\end{rem:endintro} 

\subsection{Some notation and terminology}\label{subsub:notterm} In order to avoid as much as possible having to repeatedly separate proofs according as to whether $N$ is even or odd, we introduce the following notation. If $N$ is an integer and $N\geq 2g+1$, we write
\begin{equation*}
\mu=\mu_N=\Big\lfloor{\frac{N-2g-1}{2}}\Big\rfloor,     \qquad   \nu=\nu_N=\Big\lfloor{\frac{N}{2}}\Big\rfloor+1.
\end{equation*}
The significance of the quantities will become clearer after \case{a} of Lemma~\ref{lem:first}. The following trivial properties will be used repeatedly. 

\begin{lemma}\label{lem:not} We have $\nu+\mu=N-g$ and
\begin{equation*}
\nu-\mu=\begin{cases} g+1 &\text{if $N$ is odd}\\
g+2  &\text{if $N$ is even}. \end{cases}
\end{equation*}
\end{lemma}

In \S\S~\ref{sec:HSanduniv}, \ref{sec:hyperelliptic} and the beginning of \S~\ref{sec:Cantorprop} we shall mostly be concerned with matrices such as $M_N(x)$ which have $g-1$ more columns than rows, $g\geq 1$ being the genus of the curve.  By a \emph{submatrix} of a such a matrix, we shall always understand a \emph{square} submatrix of the largest possible size, in other words obtained by the suppression of $g-1$ columns.  A \emph{subdeterminant} is the determinant of a submatrix as above. The \emph{leftmost} submatrix is that obtained by suppressing the $g-1$ columns on the right and the leftmost subdeterminant is the determinant of the leftmost submatrix. Similarly for the rightmost submatrix and subdeterminant. 

\section{Hasse-Schmidt derivatives and universal polynomials}\label{sec:HSanduniv}

\subsection{Hasse-Schmidt derivatives} Let $\Lambda$ be a commutative domain and let $D_1$ be the usual derivation on the Laurent polynomial ring $\Lambda[x,\frac{1}{x}]$, so that $D_1(x^n)=nx^{n-1}$ and $D_1$ vanishes on $\Lambda$.  Recall that the Hasse-Schmidt derivatives associated to  $D_1$ are the members of the unique sequence $D=(D_n)_{n\geq 0}$ of $\Lambda$-linear operators on $\Lambda[x]$ such that $D_0$ is the identity, $D_1$ is as above and $D_n(x^m)= \binom{m}{n}x^{m-n}$ for all $n\geq 1$ and for all $m$. They satisfy the Leibniz rule $D_n(uv)=\sum_{\ell =0}^nD_\ell(u)D_{n-\ell}(v)$ for all $u$, $v\in \Lambda[x]$.  Also, $D_mD_n=\binom{m+n}{n}D_{m+n}$ for all $m$, $n$.  Finally, $D_1^n=n! D_n$ for all $n$. 

\subsection{Universal polynomials} Fix an integer $g\geq 1$. As in the Introduction, let $\{p_i\}_{0\leq i\leq 2g}$ and $\{q_i\}_{0\leq i\leq g}$ be two sets of variables and let $\Lambda=\bZ[p_0,\dots, p_{2g},q_0,\dots q_g]$ be the polynomial ring. Define  $\Lambda_2=\Lambda\otimes \bZ[\frac{1}{2}]$ and $\Lambda_\bQ=\Lambda\otimes \bQ$, so we have inclusions $\Lambda\subseteq \Lambda_2\subseteq \Lambda_\bQ$. For every $n\geq 0$, we denote by $\sD_n$ the derivation $D_n$ associated to these rings. Set $\sP(\sx)=p_0+p_1\sx+\cdots  +p_{2g}\sx^{2g}+\sx^{2g+1}$, $\sQ(\sx)=q_0+q_1\sx+ \cdots   +q_g\sx^g$ and put $\mysf(\sx)=\sP(\sx)+\frac{1}{4}\sQ(\sx)^2$, so that $\sP$, $\sQ\in \Lambda[\sx]$ and $\mysf\in \Lambda_2[\sx]$. Consider the ring $\Lambda[\sx,\sy]$ where $\sy$ satisfies $\sy^2+\sQ(\sx)\sy=\sP(\sx)$ and write $\sz=\sy+\frac{1}{2}\sQ(\sx)$, an element of $\Lambda_2[\sx,\sy]$. Then $\sz^2=\mysf(\sx)$. Since $\Lambda$ has characteristic zero, $\sD_1$ extends uniquely to a derivation on $\Lambda[\sx,\sy]$. The same is true for $\sD$, and in fact $\sD_n=\frac{\sD_1^n}{n!}$ for all $n$.   Similarly, $\sD$ extends uniquely to $\Lambda_2[\sx,\sy]$ and $\Lambda_{\bQ}[\sx,\sy]$.  
To simplify notation, write $\sy_n=\sD_n\sy$ and $\sz_n=\sD_n\sz$ for all $n$.

\begin{lemma}\label{lem:degbound} Let $i$ be an integer, $i\geq 0$. 

\case{a} For all integers $n\geq 1$ there exist unique polynomials $\myss_{i,n}(\sx)$,  $\st_{i,n}(\sx)\in \Lambda[\sx]$ and $\su_{i,n}(\sx)\in \Lambda_2[x]$ such that 
\begin{equation*}
\sD_n(\sx^i\sy)=\frac{\myss_{i,n}(\sx)+\st_{i,n}(\sx)\sy}{(2\sz)^{2n-1}}  \quad \text{and}  \quad  \sD_n(\sx^i\sz)=\frac{\su_{i,n}(\sx)}{(2\sz)^{2n-1}}.
\end{equation*}

\case{b} If $n>i+g$, then $\sD_n(\sx^i\sy)=\sD_n(\sx^i\sz)$ and hence $\myss_{i,n}(\sx)=\su_{i,n}(\sx)$ and $\st_{i,n}(\sx)=0$. In particular, $\su_{i,n}(\sx)\in \Lambda$ whenever $n>i+g$.
\end{lemma}

\begin{proof} \case{a} The uniqueness of $\myss_{i,n}(\sx)$, $\st_{i,n}(\sx)$ and $\su_{i,n}(\sx)$ is clear.  We prove the existence of $\myss_{i,n}(\sx)$ and $\st_{i,n}(\sx)$ and leave the case of $\su_{i,n}(\sx)$ to the reader.  

Suppose $n=1$. From $\sy^2+\sQ(\sx)\sy=\sP(\sx)$, we find that $(2\sy+\sQ(\sx))\sy_1+(\sD_1\sQ(\sx))\sy=\sD_1\sP(\sx)$, so $y_1=\frac{\sD_1\sP(\sx)-(\sD_1\sQ(\sx))\sy}{2\sy+\sQ(\sx)}=\frac{\sD_1\sP(\sx)-(\sD_1\sQ(\sx))\sy}{2\sz}$, so $\myss_{0,1}(\sx)=\sD_1\sP(\sx)$ and $\st_{0,1}(\sx)=-\sD_1\sQ(\sx)$. These both lie in $\Lambda[x]$, so the result is true for $i=0$. If $i>0$, we have
\begin{align*}
\sD_1(\sx^i\sy)&=\sx^i\sD_1\sy+i\sx^{i-1}\sy=\frac{\sx^i(\sD_1\sP(\sx)-\sD_1\sQ(\sx)\sy)+i\sx^{i-1}2\sz\sy}{2\sz}\\
&=\frac{\sx^i(\sD_1\sP(\sx)-\sD_1\sQ(\sx)\sy)+i\sx^{i-1}(\sP(\sx)-2\sQ(\sx)\sy)}{2\sz},
\end{align*}
so $\myss_{i,1}(\sx)=\sx^i\sD_1\sP(\sx)+i\sx^{i-1}\sP(\sx)$ and $\st_{i,1}(\sx)=-\sx^i\sD_1\sQ(\sx)-2i\sx^{i-1}\sQ(\sx)$ both lie in $\Lambda[\sx]$.  

Next note that if $m\geq 2$ and the result is known for $(i,n)$ for $i=0$ and for all $n<m$, then we can deduce it for $i=0$ and $n=m$ as follows. From $\sy^2+\sQ(\sx)\sy=\sP(\sx)$ we deduce that $\sD_m\sP(\sx)=\sD_m((\sy+\sQ(\sx))\sy)=\sum_{k=0}^m\sD_k(\sy+\sQ(\sx))\sy_{m-k}$. The sum of the terms where $k=0$ and $k=m$ is $(2\sy+\sQ(\sx))\sy_m+(\sD_m\sQ(\sx))\sy=2\sz\sy_m+(\sD_m\sQ(\sx))\sy$ and if $1\leq k\leq m-1$ then by hypothesis $\sD_k(\sy+\sQ(\sx))\sy_{m-k}$ can be written as a product $\frac{\alpha}{(2\sz)^{2k-1}}\frac{\beta}{(2\sz)^{2(m-k)-1}}=\frac{\alpha\beta}{(2\sz)^{2m-2}}$ for some $\alpha$, $\beta\in \Lambda$.   Adding all these terms and rearranging proves the assertion.

Finally, if the result is known for all $(i,n)$ with $n<m$ and for $(0,m)$, it follows easily for all $i$ when $n=m$ using the identity $\sD_m(\sx^i\sy)=\sum_{k=0}^m\binom{i}{k}\sx^{i-k}\sy_{m-k}$. 

\case{b} Since $\sQ$ has degree $g$, $\sx^i\sQ(\sx)$ has degree $g+i$, so $\sD_n(\sx^i\sQ(\sx))=0$ whenever $n>g+i$.  Since $\sx^i\sz=\sx^i\sy+\frac{1}{2}\sx^i\sQ(\sx)$, this implies that $\sD_n(\sx^i\sy)=\sD_n(\sx^i\sz)$ as claimed. 
\end{proof}

Let $\{b_i\}_{i\geq 0}$ be further variables and let $B=\Lambda[b_0,b_1,\dots  ]$ be the polynomial ring. Then again,  $\sD$ has a unique extension to $B[\sx,\sy]$ vanishing on $B$, which we denote again by $\sD$. Let $\Psi(\sx)=b_0+b_1\sx+\cdots  +b_m\sx^m\in B[\sx]$. It follows from Lemma~\ref{lem:degbound} that for each $n\geq 0$ there is a unique polynomial $\sr_{\Psi,n}(\sx)$ such that $\sD_n(\Psi(\sx)\sz)=\frac{\sr_{\Psi,n}(\sx)}{(2\sz)^{2n-1}}$.  Here $\sr_{\Psi,0}(\sx)=\frac{1}{2}\Psi(\sx)$ lies in $B[\frac{1}{2},\sx]$ and $\sr_{\Psi,n}(\sx)\in B[\sx]$ if $n\geq 1$.  We want to compute the degree of $\sr_{\Psi,n}(\sx)$ as well as its leading coefficient. Since we are working in characteristic zero, we can use induction on the degree $n$ of the higher ordinary derivative $\sD_1^n=n!\sD_n$.  Write $\tilde{\sr}_{\Psi,n}=n!\sr_{\Psi,n}$, so by what has just been said, $\sD_1^n(\Psi(\sx)\sz)=\frac{\tilde{\sr}_{\Psi,n}(\sx)}{(2\sz)^{2n-1}}$.

For integers $m$, $n$ with $n\geq 0$ define $C_{m,n}=\prod_{r=0}^{n-1}(m-2r)$, so $C_{m,0}=1$ for all $m$. In our applications, $m$ will always be odd, so $C_{m,n}\neq 0$.

\begin{lemma}\label{lem:Cmn} \case{a} For all $n\geq 0$ we have
\begin{equation*}
\tilde{\sr}_{\Psi,n+1}(\sx)=2(2\sD_1\tilde{\sr}_{\Psi,n}(\sx)\mysf(\sx)+(1-2n)\tilde{\sr}_{\Psi,n}(\sx) \sD_1\mysf(x)).
\end{equation*}

\case{b} If $\Psi(\sx)=b_0+b_1\sx+\cdots +b_m\sx^m$ as above then $\tilde{\sr}_{\Psi,n}$ has degree $2gn+m$ and leading coefficient $2^{n-1}b_mC_{2g+2m+1,n}$. 
\end{lemma}

\begin{proof}
\case{a} Use the fact that $\sD_1^{n+1}(\Psi(\sx)\sz)=\sD_1(\sD_1^n(\Psi(\sx)\sz))$ together with $\sD_1\sz=\frac{\sD_1\mysf(\sx)}{2\mysf(\sx)}\sz$.  One proves \case{b} by an easy induction using \case{a} and observing that the terms of degree at most $2g$ in $\mysf$ and of degree at most $m-1$ in $\Psi$ do not impact the coefficient of degree $2gn+m$ in $\tilde{\sr}_{\Psi,n}$.  
\end{proof}

\subsection{The universal matrix $\sM_{\sP,\sQ,N}(\sx)$}\label{par:univmat} Fix an integer $N\geq 2g+1$ and write $\nu=\nu_N=\lfloor\frac{N}{2}\rfloor+1$,  $\mu=\mu_N=\lfloor\frac{N-2g-1}{2}\rfloor$. We define the universal matrix
\begin{equation*}
\sM_{N}(\sx)=\begin{pmatrix} \myss_{0,\nu}  & \myss_{0,\nu+1} & \cdots  & \myss_{0,N-1} \\ 
\myss_{1,\nu} & \myss_{1,\nu+1} & \cdots  & \myss_{1,N-1} \\
\vdots   &  \vdots       & \ddots  &    \vdots   \\
\myss_{\mu, \nu}& \myss_{\mu,\nu+1}  &  \cdots  & \myss_{\mu, N-1} \end{pmatrix}.
\end{equation*}
Thus $\sM_N(\sx)$ has $\mu+1$ rows and $\mu+g$ columns and, by Lemma~\ref{lem:degbound}, it has coefficients in $\Lambda[\sx]$.

Denote by $\cS_0$ the set $\{j=(j_1,j_2,\dots, j_{\mu+1})\in \bZ^{\mu+1}\mid \nu\leq j_1<j_2<\cdots < j_{\mu+1}\leq N-1\}$. For each $j\in \cS_0$ denote by $\Sigma_j(\sx)$ the submatrix of $\sM_{N}(x)$ formed by columns $j_r-\nu+1$, $r\in \{1,\dots, \mu+1\}$. In other words
\begin{equation*}
\Sigma_j(\sx)=\begin{pmatrix} \myss_{0,j_1} & \myss_{0,j_2}  &    \cdots  &  \myss_{0,j_{\mu+1}}\\
 \myss_{1,j_1} & \myss_{1,j_2}  &    \cdots  &  \myss_{1,j_{\mu+1}}\\
 \vdots     &  \vdots      &    \ddots  &   \vdots            \\
 \myss_{\mu,j_1} & \myss_{\mu,j_2}  &    \cdots  &  \myss_{\mu,j_{\mu+1}} \end{pmatrix}.
\end{equation*}

\begin{lemma} \label{lem:degSigma}
$\det{\Sigma_j(\sx)}$ is a polynomial of degree at most $2g\sum_{\ell=1}^{\mu+1}j_\ell+\frac{\mu(\mu+1)}{2}$. 
\end{lemma}

\begin{proof}
The entry of $\Sigma_j(\sx)$ at the intersection of row $i$ and column $\ell$ is $\myss_{i,j_\ell}(\sx)$, which is a polynomial of degree at most $2gj_\ell+i$. Thus if $\pi$ is a permutation of $\{0,1,\dots, \mu\}$, then the product $\pm \prod_{i=0}^{\mu}\myss_{i,j_{\pi(i)+1}}(\sx)$ appearing in  the development of $\det{\Sigma_j(\sx)}$ has degree at most 
$\sum_{i=0}^{\mu}{(2gj_{\pi(i)+1}+i)}=2g\sum_{i=1}^{\mu+1}j_i+\frac{\mu(\mu+1)}{2}$, as required.
\end{proof}

Let $M$ be a matrix whose entries $M_{i,j}$ are polynomials in one variable $\sx$ of degree at most $d_{i,j}$. We define the \emph{leading coefficient matrix} associated to $M$ to be the matrix whose $(i,j)^{\text{th}}$ entry is the coefficient of $\sx^{d_{i,j}}$ in $M_{i,j}$.  This depends of course on the choice of $d_{i,j}$; here we apply it to the matrix $\sM_{N}(\sx)$ defined above with $d_{i,j}$ equal to $2g(\nu+j-1)+i$. Indeed, this is the bound on the degree of $\myss_{i,\nu+j-1}$ obtained by applying Lemma~\ref{lem:degbound}. Since $\sr_{\Psi,n}=\frac{1}{n!}\tilde{\sr}_{\Psi,n}$, we deduce the following from \case{b} of Lemma~\ref{lem:Cmn}. 

\begin{lemma}\label{lem:leadcoeffmat} The $(i,j)^{\text{th}}$ entry of the leading coefficient matrix associated to $\sM_{N}(\sx)$ is
\begin{equation*} 2^{\nu+j-2}\frac{C_{2(g+i)-1,\nu+j-1}}{(\nu+j-1) !}.
\end{equation*}
\end{lemma}

Let $\Gamma_{N}(\sx)$ be the determinant of the leftmost $(\mu+1)\times (\mu+1)$ submatrix of $\sM_{N}(\sx)$ (see \ref{subsub:notterm}), in other words $\Gamma_{N}(\sx)=\det{\Sigma_j(\sx)}$ where $j=(\nu,\nu+1,\dots, \nu+\mu)$.  

\begin{prop} \label{prop:valleadcoeffs}
The degree of $\Gamma_{N}(\sx)$ is $g((N-g)(N-g+1)-\nu(\nu-1))+\frac{\mu(\mu+1)}{2}$ and its leading coefficient is equal to
\begin{equation*}
2^{(\nu+\mu-1)(\mu+1)} \frac{\left(\prod_{i=1}^{\mu+1}{C_{2g+2i-1,\nu}}\right)\left(\prod_{j=1}^{\mu+1}(j-1)!\right)} {\prod_{j=1}^{\mu+1}(\nu+j-1)!}
\end{equation*}
\end{prop}

\begin{proof} The fact that the degree is at most $g((N-g)(N-g+1)-\nu(\nu-1))+\frac{\mu(\mu+1)}{2}$ follows from Lemma~\ref{lem:degSigma}, so we need to compute the coefficient of that degree and check that it is non-zero. The corresponding leading coefficient submatrix (which we denote by $\sM^c_{\sP,\sQ,N}$) can be read off from Lemma~\ref{lem:leadcoeffmat} and the desired coefficient is equal to its determinant.  To compute it, recall that if $c\in \bQ$ and if
\begin{equation*}
T_n=\begin{pmatrix} 1 & \binom{c}{1} & \binom{c}{2} & \cdots   &  \binom{c}{n-1} \\  
1 & \binom{c+1}{1} & \binom{c+1}{2} & \cdots   &   \binom{c+1}{n-1}   \\
\vdots &  \vdots  &  \vdots   &  \ddots     &    \vdots          \\
1 & \binom{c+n-1}{1} & \binom{c+n-1}{2} & \cdots   &  \binom{c+n-1}{n-1} 
\end{pmatrix},
\end{equation*}
then $\det{T_n}=1$.  This is clear for $n=1$. If $n> 1$ subtract row $n-i-1$ from row $n-i$ for $i=0$, $1$, \dots, $n-2$ and use the identity $\binom{u}{m}-\binom{u-1}{m}=\binom{u-1}{m-1}$ to show that $\det{T_n}=\det{T_{n-1}}$.   Using  the identity $C_{m,\ell+\ell'}=C_{m,\ell}C_{m-2\ell,\ell'}$, we see that  
\begin{align*}
2^{\nu+j-2}\frac{C_{2g+2i-1,\nu+j-1}}{(\nu+j-1)!}&=2^{\nu+j-2}\frac{C_{2g+2i-1,\nu}}{(\nu+j-1)!}C_{2g+2i-2\nu-1,j-1}\\
&=2^{\nu+j-2}\frac{C_{2g+2i-1,\nu}}{(\nu+j-1)!}2^{j-1}(j-1)!\binom{c+i}{j-1},
\end{align*}
where $c=g-\nu-\frac{1}{2}$ is independent of $i$ and $j$.  We deduce that $\sM^c_{N}$ is obtained from $T_{\mu+1}$ by multiplying row $i$ by $C_{2g+2i-1,\nu}$ for all $i$ and multiplying column $j$ by $\frac{2^{\nu+2(j-1)}(j-1)!}{(\nu+j-1)!}$ for all $j$. Finally the exponent of $2$ is $\sum_{j=1}^{\mu+1}(\nu+2j-3)=(\nu+\mu-1)(\mu+1)$ and the result follows. 
\end{proof}

\subsection{Powers of $\sF(\sx)$ dividing $\det{\Sigma_j(\sx)}$}  Write $\sF(\sx)=4\mysf(\sx)$, so that $\sF(\sx)\in \Lambda[\sx]$. The main result of this subsection is the following Proposition.

\begin{prop} \label{prop:powfxdivSigma}
For all $j\in \cS_0$, $\sF(\sx)^{\mu(\mu+1)/2}$ divides $\det{\Sigma_j(\sx)}$.
\end{prop}

Before giving the proof, we define the polynomials $\Pi_{j,N}(\sx)$ and, in particular, the polynomial $\Delta_{N}(\sx)$ mentioned just before Theorem~\ref{th:helltorsionbound}.

\begin{defnPiDelta} \label{defPiDelta} We define $\Pi_{j,N}(\sx)=\det{(\Sigma_j(\sx))}/\sF(\sx)^{\mu(\mu+1)/2}$ and 
\begin{equation*}
\Delta_{\sP,\sQ,N}(\sx)=\Gamma_{N}(\sx)/\sF(\sx)^{\mu(\mu+1)/2}.
\end{equation*}
\end{defnPiDelta}
As before, we abbreviate $\Delta_{\sP,\sQ,N}(\sx)$ to $\Delta_N(\sx)$.

Since $\sF(\sx)$ has degree $2g+1$, $\sF(\sx)^{\mu(\mu+1)/2}$ has degree $(2g+1)\mu(\mu+1)/2$. Also, the leading coefficient of $\sF(\sx)$ is $4$, so that of $\sF(\sx)^{\mu(\mu+1)/2}$ is $2^{\mu(\mu+1)}$. Recalling that $\nu=\lfloor\frac{N}{2}\rfloor+1$ and  $\mu=\lfloor\frac{N-2g-1}{2}\rfloor$, we see that the following corollaries follow from Propositions~\ref{prop:valleadcoeffs} and \ref{prop:powfxdivSigma} together with Lemma~\ref{lem:degSigma}. 

\begin{cor} \label{cor:degDelta}
Suppose $N\geq 2g+1$ and let $\delta_{N,g}$ be as in the statement of Theorem~\ref{th:helltorsionbound}. Then $\deg{(\Delta_{N}(\sx))}=\delta_{N,g}$ and its leading coefficient is equal to 
\begin{equation*}
2^{(\nu-1)(\mu+1)} \frac{\left(\prod_{i=1}^{\mu+1}{C_{2g+2i-1,\nu}}\right)\left(\prod_{j=1}^{\mu+1}(j-1)!\right)} {\prod_{j=1}^{\mu+1}(\nu+j-1)!}
\end{equation*}
\end{cor}

\begin{cor} \label{cor:degPij}
For all $j\in \cS_0$, the degree of $\Pi_{j,N}(\sx)$ is at most equal to $2g\sum_{r=1}^{\mu+1}j_r-g\mu(\mu+1)$. 
\end{cor}

\begin{proof}[Proof of Proposition~\ref{prop:powfxdivSigma}]  In what follows, $i$ is an integer such that $0\leq i\leq \mu$. Furthermore, all the indices $m$ that appear satisfy $m\geq \nu$, and hence $m>i+g$, since $i\leq \mu<\nu\leq m$ and $\nu-\mu\in \{g+1,g+2\}$. Now from
\begin{equation*}
\sD_m(\sx^i\sy)=\sum_{k=0}^m\sD_{m-k}\sx^{i}\sD_k\sy=\sum_{k=m-i}^m\binom{i}{m-k}\sx^{i-k+m}\sy_k
\end{equation*}
we deduce
\begin{equation*}
\myss_{i,m}(\sx)=\sum_{k=m-i+1}^m\binom{i}{m-k}\sx^{i-k+m}(\sF(\sx))^{m-k}\myss_{k}(\sx)+(\sF(\sx))^{i}\myss_{m-i}(\sx),
\end{equation*}
whence
\begin{equation} \label{eq:myss}
\myss_{i,m}(\sx)=\sum_{\ell=1}^{i} \binom{i}{\ell}\sx^{\ell}(\sF(\sx))^{i-\ell}\myss_{m-i+\ell}(x)+ (\sF(\sx))^{i}s_{m-i}(\sx),
\end{equation}
where we abbreviate $\myss_{0,k}$ to $\myss_k$.  Note that the coefficients $\binom{i}{\ell}x^{\ell}(\sF(\sx))^{i-\ell}$ of $\myss_{m-i+\ell}(\sx)$ and $\sF(\sx)^i$ of $\myss_{m-i}(\sx)$ are independent of $m$.  Rewriting $\Sigma_j(x)$ using (\ref{eq:myss}) and applying suitable row operations, this implies that $\det{\Sigma_j(\sx)}$ is equal to the determinant of the matrix
 \begin{equation*}
 \begin{pmatrix}
 \myss_{j_1}(\sx) & \myss_{j_2}(\sx)  &  \cdots     &   \myss_{j_{\mu+1}}(\sx) \\
 \sF(\sx)\myss_{j_1-1}(\sx) & \sF(x)\myss_{j_2-1}(\sx)  &  \cdots     &   \sF(\sx)\myss_{j_{\mu+1}-1}(\sx) \\
 \vdots      &  \vdots        & \ddots     &    \vdots    \\
 \sF(\sx)^\mu \myss_{j_1-\mu}(\sx) & \sF(\sx)^\mu \myss_{j_2-\mu}(\sx)  &  \cdots     &   \sF(\sx)^\mu \myss_{j_{\mu+1}-\mu}(\sx) 
 \end{pmatrix}.
 \end{equation*}
Since $\myss_m(\sx)\in \Lambda[\sx]$,  $\sF(\sx)\in \Lambda[\sx]$, and the rows are divisible successively by $1$, $\sF(\sx)$, $\sF(\sx)^2$, \dots, $\sF(\sx)^\mu$, we deduce that this determinant is divisible by $\sF(\sx)^{1+2+\cdots +\mu}=\sF(\sx)^{\mu(\mu+1)/2}$ as required.
\end{proof}

The proof gives the following explicit formula for $\Pi_{j,N}(x)$. 

\begin{cor}  \label{cor:pidet} Let $j\in \cS_0$ and write
\begin{equation*}
\sS_j(\sx)=\begin{pmatrix} \myss_{j_1}(\sx) &  \myss_{j_2}(\sx)  & \cdots   &   \myss_{j_{\mu+1}}(\sx) \\
\myss_{j_1-1}(\sx) &  \myss_{j_2-1}(\sx)  & \cdots   &   \myss_{j_{\mu+1}-1}(\sx) \\
\vdots      &  \vdots        & \ddots     &    \vdots    \\
\myss_{j_1-\mu}(\sx) & \myss_{j_2-\mu}(\sx)  &  \cdots     &   \myss_{j_{\mu+1}-\mu}(\sx) \end{pmatrix}.
\end{equation*}
Then $\Pi_{j,N}(\sx)=\det{\sS_j(\sx)}$. 
\end{cor}

\subsection{Relation between $\Delta_N$ and Cantor's $P_{N-g+1}$}\label{subsec:Cantor}  We first recall Cantor's definition of the polynomials he denotes by $P_r$ when $r\geq g$. We use notation close to that of \cite{Ca94}, indicating the modifications made to avoid confusion with the notation of the present paper.   Write $\mysf(\sx)=f_0+f_1\sx+\cdots +f_{2g}\sx^{2g}+\sx^{2g+1}$ where we think of the coefficients $f_i$ as independent variables. Let $\zeta$ be another variable (which Cantor denotes by $z$) and write $E(\zeta)=\mysf(\sx-\zeta)$, so $E(\zeta)=\sum_{i=0}^{2g+1}e_i\zeta^i$ with $e_i\in \bQ[f_0,f_1,\cdots f_{2g}][\sx]$. Note that $e_0=\mysf(\sx)$  and $e_{2g+1}=-1$. Let $\sz^2=\mysf(\sx)$ (this is our $\sz$ that corresponds Cantor's $y$) and $E_1(\zeta)=(E(\zeta)-\sz^2)/\zeta\in \Lambda[\sx,\zeta]$ and define a square root $S(\zeta)$ of $E(\zeta)$ by
\begin{equation*}
S(\zeta)=(-1)^{g+1}\sz\sqrt{1+\zeta\frac{E_1(\zeta)}{\sz^2}}.
\end{equation*}
Here $\sqrt{1+Z}=1+\frac{1}{2}Z+\cdots+ \binom{1/2}{n}Z^n+\cdots $ is the usual binomial expansion.  Let $\sigma_n$ be the coefficient of $\zeta^n$ in $S(\zeta)$ (Cantor writes it as $s_n$).  Writing $\myss_n$ for our $\myss_{0,n}$ as before, we see that $\sigma_n=(-1)^n\sD_nS(0)=(-1)^{n+g+1}\frac{\myss_n}{(2\sz)^{2n-1}}$.  Alternatively, writing $\sz_n=\sD_n\sz$, we have $\sigma_n=(-1)^{n+g+1}\sz_n$. 

For $r\geq g$ (we only need $r\geq g+2$), Cantor writes $m_r=\lfloor{\frac{r+g}{2}}\rfloor$ and $n_r=\lfloor{\frac{r-g-1}{2}}\rfloor$. We take $r=N-g+1$. Then $m_{r+1}=\nu$ and $n_{r+1}=\mu+1$, and Cantor defines the Hankel matrix (see his equation (3.1))
\begin{equation*}
H_{\nu,\mu+1}=\begin{pmatrix}  \sigma_{\nu-\mu} & \sigma_{\nu-\mu+1} &  \cdots   &   \sigma_{\nu} \\
        \sigma_{\nu-\mu+1} & \sigma_{\nu-\mu+2} &  \cdots   &   \sigma_{\nu+1}      \\
        \vdots                       &  \vdots                      &  \ddots   &       \vdots               \\
        \sigma_\nu               & \sigma_{\nu+1}        &  \cdots    &  \sigma_{\nu+\mu}                                               \end{pmatrix}
\end{equation*}
and then defines (see his equation (8.7)) 
\begin{equation*}
\psi_r=\psi_{N-g+1}=(2\sz)^{(r^2-r-g^2+g)/2}\det{H_{\nu,\mu+1}}=(2\sz)^{N(N-2g+1)/2}\det{H_{\nu,\mu+1}}
\end{equation*}
and finally $P_r(\sx)=\psi_r$ if $N$ is odd and $P_r(x)=\psi_r/(2\sz)^{g}$ if $N$ is even (see (8.16) on page 134 of \cite{Ca94}).  Cantor proves that $P_r(\sx)$ is actually a polynomial in $\sx$, and this will also follow from our calculations.  In any case, 
\begin{align*}
\psi_{N-g+1}=&(-1)^{(g+1+\nu-\mu)(\mu+1)}(2\sz)^{N(N-2g+1)/2} \\
&\times \det{\begin{pmatrix} \sz_{\nu-\mu} & -\sz_{\nu-\mu+1} & \cdots    & (-1)^\mu \sz_{\nu}\\
-\sz_{\nu-\mu+1}& \sz_{\nu-\mu+2} &   \cdots   &   (-1)^{\mu+1} \sz_{\nu+1} \\
\vdots              &   \vdots             &    \ddots  &                 \vdots           \\ 
(-1)^{\mu}\sz_{\nu} & (-1)^{\mu+1}\sz_{\nu+1}  &  \cdots   & \sz_{\nu+\mu} \end{pmatrix}},
\end{align*}
where the sign is constant on each antidiagonal and alternates $+$, $-$, $+$, \dots{} from one antidiagonal to the next.  Removing the signs does not change the determinant, so the total sign is $(-1)^{(g+1+\nu-\mu)(\mu+1)}$ and Cantor's definition boils down to

\begin{defn} \label{def:cantor} Let $N\geq 2g+1$, $r=N-g+1$.  Define $\psi_r=\psi_{N-g+1}$ to be $(-1)^{(g+1+\nu-\mu)(\mu+1)}(2\sz)^{N(N-2g+1)/2}\det{V}$, where 
\begin{equation*}
V=\begin{pmatrix} \sz_{\nu-\mu} & \sz_{\nu-\mu+1} &  \cdots  & \sz_{\nu} \\
\sz_{\nu-\mu+1 } &   \sz_{\nu-\mu+2}     &      \cdots    & \sz_{\nu+1}             \\
\vdots    &      \vdots     &           \ddots&   \vdots                \\
\sz_{\nu}  &    \sz_{\nu+1}    &    \cdots    &  \sz_{N-g}   \end{pmatrix} .
\end{equation*}
Then define $P_{r}$ by $P_{r}=\psi_{r}$ if $N$ is odd and $P_{r}=\psi_{r}/(2\sz)^g$ if $N$ is even.
\end{defn}

\begin{prop}
Let $N\geq 2g+1$, and write $\delta_N=\lfloor{\frac{N}{2}}\rfloor$. Then $P_{N-g+1}(\sx)$ is a polynomial in $\sx$ and in fact 
\begin{equation*}
P_{N-g+1}(\sx)=(-1)^{(1-\delta_N)(\mu+1)+\lfloor{(\mu+1)/2}\rfloor}\Delta_N(\sx).
\end{equation*}
\end{prop}

\begin{proof}
Let $V'$ be the matrix obtained from $V$ by interchanging rows placed symmetrically around the middle, so that
\begin{equation*}
V'=\begin{pmatrix} \sz_{\nu} & \sz_{\nu+1} &  \cdots   &  \sz_{N-g}\\
\sz_{\nu-1} & \sz_{\nu} &  \cdots   &  \sz_{N-g-1}\\
\vdots      &  \vdots  &   \ddots &  \vdots  \\
\sz_{\nu-\mu} & \sz_{\nu-\mu+1} &  \cdots  & \sz_{\nu} \end{pmatrix}. 
\end{equation*}
Then $\det{V'}=(-1)^{\lfloor{(\mu+1)/2}\rfloor}\det{V}$ since $V'$ is obtained from $V$ by a permutation of rows that is a product of $\lfloor{(\mu+1)/2}\rfloor$  transpositions. 

Since $\sz_m=\frac{\myss_m(\sx)}{(2\sz)^{2m-1}}=\myss_m(\sx)\frac{2\sz}{\sF(\sx)^m}$, we have
\begin{equation*}
\det{V'}=\det{\begin{pmatrix}  \myss_{\nu}(\sx)\frac{2\sz}{\sF(\sx)^\nu}&  \myss_{\nu+1}(\sx)\frac{2\sz}{\sF(\sx)^{\nu+1}} & \cdots   &   \myss_{\nu+\mu}(\sx)\frac{2\sz}{\sF(\sx)^{\nu+\mu}} \\
\myss_{\nu-1}(\sx)\frac{2\sz}{\sF(\sx)^{\nu-1}}&  \myss_{\nu}(\sx)\frac{2\sz}{\sF(\sx)^{\nu}} & \cdots   &   \myss_{\nu+\mu-1}(\sx)\frac{2\sz}{\sF(\sx)^{\nu+\mu-1}}  \\
\vdots                                           &   \vdots                         &  \ddots                  &  \vdots     \\
\myss_{\nu-\mu}(\sx)\frac{2\sz}{\sF(\sx)^{\nu-\mu}}&  \myss_{\nu-\mu+1}(\sx)\frac{2\sz}{\sF(\sx)^{\nu-\mu+1}} & \cdots   &   \myss_{\nu}(\sx)\frac{2\sz}{\sF(\sx)^{\nu}} \end{pmatrix}}.
\end{equation*}
If $i$, $\ell\in \{1,2,\dots, \mu+1\}$, the power of $\sF(\sx)$ in the denominator of the entry at the intersection of the $i^{\text{th}}$ row and the $\ell^{\text{th}}$ column is $\sF(\sx)^{\nu+i-\ell}$. Hence in the development of the determinant, the power of $\sF(\sx)$ in the denominator of the product corresponding to the permutation $\pi$ of $\{1,2,\dots, \mu+1\}$ is $\sF(\sx)^{\sum_{i=1}^{\mu+1}\nu+i-\pi(i)}=\sF(\sx)^{\nu(\mu+1)}$. Since $\sF(\sx)=(2\sz)^2$, it follows that
\begin{equation*}
\det{V'}=(2\sz)^{-(2\nu-1)(\mu+1)}\det{\begin{pmatrix}  \myss_{\nu} & \myss_{\nu+1} & \cdots   &    \myss_{\nu+\mu} \\
\myss_{\nu-1} & \myss_{\nu} & \cdots   &    \myss_{\nu+\mu-1} \\
\vdots     &    \vdots   &  \ddots  &   \vdots          \\
\myss_{\nu-\mu} & \myss_{\nu-\mu+1} & \cdots   &  \myss_{\nu} \end{pmatrix}}
\end{equation*}
and this is equal to $(2\sz)^{-(2\nu-1)(\mu+1)}\Delta_N(\sx)$ by Corollary~\ref{cor:pidet}. (The matrix on the right is just the matrix $\sS_j(\sx)$ with $j=(\nu,\nu+1, \dots,  \nu+\mu)$.) Finally
\begin{equation*}
(2\nu-1)(\mu+1)=\begin{cases} N(N-2g+1)/2  & \text{if $N$ is odd}\\
                                       N(N-2g+1)/2-g          & \text{if $N$ is even}\\
                                       \end{cases}
\end{equation*}
so that from Definition~\ref{def:cantor} and Lemma~\ref{lem:not},
 \begin{align*}
P_{N-g+1}(\sx)&=(-1)^{(g+1+\nu-\mu)(\mu+1)}(2\sz)^{(2\nu-1)(\mu+1)}\det{V}\\ 
&=(-1)^{(g+1+\nu-\mu)(\mu+1)+\lfloor{(\mu+1)/2}\rfloor}(2\sz)^{(2\nu-1)(\mu+1)}\det{V'}\\ 
&=(-1)^{(1-\delta_N)(\mu+1)+\lfloor{(\mu+1)/2}\rfloor}\Delta_N(\sx)
\end{align*}
as claimed.
\end{proof}

\section{Hyperelliptic curves of odd degree} \label{sec:hyperelliptic} 

In this section, we apply the results of \S~\ref{sec:HSanduniv} to prove several bounds on the torsion of odd degree hyperelliptic curves, of which Theorem~\ref{th:helltorsionbound} and Proposition~\ref{prop:2g12} are special cases.

\subsection{Preliminaries} \label{sec:hyperprelims} The field $k$ is always supposed to be algebraically closed.  We work with an affine model $y^2+Q(x)y=P(x)$ of the hyperelliptic curve $X_{P,Q}$ of genus $g\geq 1$. Here $P(x)$ is monic of degree $2g+1$ and $Q(x)$ has degree at most $g$. Write $F(x)=Q(x)^2+4P(x)$. Then $(2y+Q(x))^2=F(x)$ and the  hypothesis that $X$ is smooth implies that $F\neq 0$.  Let $J_{P,Q}$ be the jacobian of $X$. We always think of $X$ as embedded in $J$ using $\infty$ as base point.  Write $K$ for the rational function field $k(x)$ and $L=k(X)$, so that $L$ is a separable quadratic extension of $K$. 
  
There is a unique ring homomorphism $\Lambda\to k$ that sends $p_i$ to the coefficient of $x^i$ in $P$ for all $i\in \{0,1,\dots, 2g\}$ and $q_i$ to the coefficient of $x^i$ in $Q$ for all $i\in \{0,1,\dots, g\}$. It extends to a homomorphism $\sigma:\Lambda(\sx)\to K$ by sending $\sx$ to $x$. In particular, $\sigma(\Lambda[\sx])\subseteq k[x]$, and $\sigma(\sP(\sx))=P(x)$,  $\sigma(\sQ(\sx))=Q(x)$ and $\sigma(\sF(\sx))=F(x)$. Since $y(y+Q(x))=P(x)$ and $P(x)\neq 0$, $\sigma$ then extends to a homomorphism $\Lambda(\sx,\sy)\to L$ with $\sigma(\sy)=y$. We use $\sigma$ to transfer concepts discussed in \S~\ref{sec:HSanduniv} to $L$.  So from now on, $D_1$ denotes the extension to $L$ of the standard derivation on $K$ (so $D_1x=1$ and $D_1\lambda=0$ for all $\lambda\in k$) and $D=(D_n)_{n\geq 0}$ the associated Hasse-Schmidt derivative.  Then clearly $\sigma\circ \sD_n=D_n\circ \sigma$ for all $n$. From now on, we denote by an italicized letter $s_{i,n}$, $M$, $\varSigma$, $\varPsi$, \dots{} the object obtained by applying $\sigma$ to the object denoted in \S~\ref{sec:HSanduniv} by the corresponding sans serif or upright greek letter $\myss_{i,n}$, $\sM$, $\Sigma$, $\Psi$, \dots{} . 

A few remarks are in order.  The definitions of $\mysf$ and $\sz$ involve division by $2$, so to avoid possible confusion in characteristic $2$ we shall not use $f$ and $z$ in this way. Note that $2\sz=2\sy+\sQ(\sx)$ which belongs to $\Lambda[\sx,\sy]$, so $\sigma(2\sz)=2y+Q(x)$ is well-defined in every characteristic (and equal to $Q(x)$ in characteristic $2$). Thus again to avoid confusion, we shall always denote $\sigma(2\sz)$ by $2y+Q(x)$. Similarly, we shall use only the notations $s_{i,n}$, $t_{i,n}$ and $u_{i,n}$ for values of $i$ and $n$ for which Lemma~\ref{lem:degbound} asserts that $\myss_{i,n}$, $\st_{i,n}$ and $\su_{i,n}$ belong to $\Lambda[x]$.  

As another example, the leading coefficient of $\Gamma_{N}(x)$ that we computed in Proposition~\ref{prop:valleadcoeffs} may now be $0$. Thus we can only conclude that the degree of $\varGamma_{N}(x)$ ($=\sigma(\Gamma_N(\sx))$) is \emph{bounded} by  $g((N-g)(N-g+1)-\nu(\nu-1))+\frac{\mu(\mu+1)}{2}$.  Similarly, by Corollary~\ref{cor:degDelta}, the degree of $\varDelta_{N}(x)$ is bounded by $\delta_{N,g}$ and not necessarily equal to it. This explains the way the conclusions of Theorem~\ref{th:helltorsionbound} have been presented. 

The following lemma is easy and left to the reader.

\begin{lemma}\label{lem:basicHS}

\case{a} Let $\lambda\in k$. Then for all $m$, $n\geq 0$ we have $D_n(x-\lambda)^m=\binom{m}{n}(x-\lambda)^{m-n}$.  

\case{b}  Let $\lambda\in k$, let $n\geq 0$ and let $(a_0,a_1,\dots, a_n)\in k^{n+1}$. Then there is a unique $\Phi(x)\in k[x]$ of degree at most $n$ such that $D_r\Phi(\lambda)=a_r$ for all $r\in \{0,1,\dots, n\}$. 

\case{c}  Let $\lambda\in k$. Then $D$ has a unique extension to $k((x-\lambda))$ that is continuous in the $(x-\lambda)$-adic topology. Denote  this extension also by $D$. 

\case{d} Let $\lambda\in k$, let $n\geq 0$ and let $\psi\in k[[x-\lambda]]$. Then $D_r\psi(\lambda)=0$ for all $r\in \{0,\dots, n-1\}$ if and only if $\psi\in (x-\lambda)^nk[[x-\lambda]]$. 
\end{lemma}

\subsection{Application to hyperelliptic curves} Denote by $\iota$ the hyperelliptic involution of $X$. The following lemma is well-known.

\begin{lemma} \label{lem:first} Suppose that $N\geq 2$. Interpret all polynomials of degree $<0$ as the zero polynomial.
  
\case{a} Every element of $\cL(N[\infty])$ can be written in a unique manner in the form $\varPhi(x)+\varPsi(x)y$, where $\varPhi$, $\varPsi\in k[x]$ and 
\begin{equation*}
\deg{\varPhi}\leq \nu-1, \qquad   \deg{\varPsi}\leq \mu.
\end{equation*}

\case{b} One has $x\in \cL(2[\infty])$ and $y\in \cL((2g+1)[\infty])$, the orders of the poles at $\infty$ being respectively $2$ and $2g+1$. 

\case{c} The action of $\iota$ on $L$ fixes $\cL(N[\infty])$ and sends $\varPhi(x)+\varPsi(x)y$ to $\varPhi(x)-\varPsi(x)(y+Q(x))$. 

\case{d} Let $\xi \in X$. Then $\iota(\xi)=\xi$ if and only if $\xi\in J[2]$. If $N\geq 2$, then $\xi\in J[N]^*$ if and only if $\iota(\xi)\in J[N]^*$.

\case{e} Let $(x_0,y_0)\in X-\{\infty\}$. Then $(x_0,y_0)\in J[2]^*$ if and only if $F(x_0)=0$.

\case{f} If $3\leq N\leq 2g$, then $X\cap J[N]^*$ is empty.

\case{g} The fixed field of the action of $\iota$ on $L$ is $K$ and the places ramified in $L/K$ are those corresponding to the points of $X\cap J[2]$.
\end{lemma}

\begin{proof} Here \case{g} follows from \case{c}, \case{d} and \case{e}.  Also, \case{f} was proved by Zarhin~\cite{Za19} (when $p\neq 2$), but at the suggestion of one of the referees we show how to deduce it from the previous assertions.   Suppose for a contradiction that $3\leq N\leq 2g$ and $\xi\in X\cap J[N]^*$. Then there exists a function $\alpha(x,y)=\varPhi(x)+\varPsi(x)y$ on $X$ as in \case{a},  with divisor $N([\xi]-[\infty])$. Since $N<2g+1$ and $2g+1$ is odd, we see using \case{b} that $\varPsi(x)=0$. Hence $\alpha(x,y)=\varPhi(x)$, so since $\alpha(\xi)=0$, we also have $\alpha(\iota(\xi))=0$. But $\xi$ is the only point of $X$ where $\alpha$ vanishes, so $\iota(\xi)=\xi$. Since $N\geq 3$, this contradicts \case{d}.
\end{proof}

Set $\cO=k[x,y][\frac{1}{F(x)}]$, so that if $(x_0,y_0)\in X$ and $F(x_0)\neq 0$, then by \case{g} $\cO$ embeds in $k[[x-x_0]]$ and $D$ operates on $\cO$ \emph{via} this embedding.

\begin{prop}\label{prop:premain} Let $\alpha\in L^\times$, let $N\geq 2g+1$ and let $\xi=(x_0,y_0)\in X-\{\infty\}$. Suppose $\xi\notin J[2]$. Then the following are equivalent.

\case{a}  $\alpha$ has divisor $N([\xi]-[\infty])$.

\case{b}  There exist non-zero $\varPsi$, $\varPhi\in k[x]$ with degrees at most $\mu$ and $\nu-1$ such that $\alpha=\varPsi(x)y+\varPhi(x)$ and $D_r\alpha(\xi)=0$ for all $r\in \{0,1,\dots, N-1\}$.
\end{prop}

\begin{proof} \case{a}$\,\Rightarrow\,$\case{b}  Since $\alpha$ has a pole of order $N$ at $\infty$ and is finite elsewhere, there exist $\varPsi$, $\varPhi$ whose degree satisfy the required inequalities. By \case{e} and \case{f} of Lemma~\ref{lem:first}, the hypotheses $N\geq 2g+1$ and $\xi\notin J[2]$ imply that neither $\varPsi$ nor $\varPhi$ can vanish, and also that $F(x_0)\neq 0$. Hence $\alpha\in \cO$ embeds in $k[[x-x_0]]$. Since $\alpha$ has a zero of order $N$ at $\xi$, $D_r\alpha(\xi)=0$ for all $r\in \{0,1,\dots, N-1\}$.

\case{b}$\,\Rightarrow\,$\case{a}   Suppose $\alpha$ has the given properties. Then $\alpha$ has a zero of order at least $N$ at $\xi$. On the other hand, the bounds of the degrees of $\varPsi$ and $\varPhi$ implies it has a pole of order at most $N$ at $\infty$, and is finite elsewhere. Since the degree of a principal divisor is zero, the divisor of $\alpha$ is necessarily $N([\xi]-[\infty])$. 
\end{proof}

\begin{theorem}  \label{th:mainth} Let $\xi =(x_0,y_0)\in X-\{\infty\}$, and let $N\geq 2g+1$. Then the following conditions are equivalent.

\case{a} $\xi \in \tilde{J}[N]$,

\case{b} There exists non-zero $\varPsi(x)\in k[x]$ of degree at most $\mu$ such that $D_r(\varPsi(x)y)(\xi)=0$ for all $r\in \{\nu, \dots, N-1\}$.
\end{theorem}

\begin{proof}. \case{a}$\,\Rightarrow\,$\case{b}. There exist $\alpha$, $\varPsi$ and $\varPhi$, as in Proposition~\ref{prop:premain}. Since the degree of $\varPhi$ is at most $\nu-1$, we have $D_r(\varPsi(x)y)(\xi)=D_r(\alpha)(\xi )=0$ for all $r\in \{\nu, \dots, N-1\}$. 

\case{b}$\,\Rightarrow\,$\case{a}. By \case{b} of Lemma~\ref{lem:basicHS}, there exists a unique polynomial $\varPhi$ of degree at most $\nu-1$ such that $D_r\varPhi(\xi )=-D_r(\varPsi y)(\xi )$ for all $r\in \{0, 1, \dots, \nu-1\}$. Set $\alpha=\varPsi(x)y+\varPhi(x)$. Then, by Proposition~\ref{prop:premain}, the divisor of $\alpha$ equals $N([\xi]-[\infty])$.  Hence $\alpha\in J[N]$. Also, $\xi \notin J[2]$, since if $\xi\in J[2]$, $\alpha$ would be a constant multiple of $(x-x_0)^{N/2}$. But since $\varPsi(x)\neq 0$, $\alpha$ is not invariant under the hyperelliptic involution, whereas $x-x_0$ is. 
\end{proof}

Let $M_{N}(x)=\sigma(\sM_N(\sx))$, where $\sigma$ is applied to each entry of $\sM_N(\sx)$ and $\sx$ is replaced by $x$. Thus the entries of $M_N(x)$ are elements of $k[x]$. Since $M_N(x)$ has $\mu+1$ rows and $\mu+g$ columns and since $g\geq 1$, its rank is at most $\mu +1$ and we shall see in a moment that this is the actual rank.

\begin{theorem}\label{th:nscJN}
Let $x_0\in k$, let $\xi \in X-\{\infty\}$, $\xi=(x_0,y_0)$ and let $N\geq 2g+1$. The following are equivalent.

\case{a} $\xi \in \tilde{J}[N]$, 

\case{b} The matrix $M_{N}(x_0)$ does not have maximal rank $\mu+1$.
\end{theorem}

\begin{proof}
\case{a}$\,\Rightarrow\,$\case{b}. Let $\varPsi$ be above, and write $\varPsi(x)=\sum_{i=0}^{\mu}b_ix^i$, where $(b_0,b_1,\dots, b_\mu)\in k^{\mu+1}$ is not $(0,0,\dots, 0)$. Then since $D_n(\varPsi(x)y)(\xi)=0$ for all $n\in \{\nu, \dots, N-1\}$, we have $r_{\varPsi,n}(x_0)=0$ for these $n$, which implies that
\begin{equation*}
\begin{pmatrix} b_0 & b_1 &  \cdots & b_\mu\end{pmatrix}M_{N}(x_0)=0,
\end{equation*}
and so $M_{N}(x_0)$ is not of maximal rank. 

\case{b}$\,\Rightarrow\,$\case{a}. If $M_{N}(x_0)$ is not of maximal rank, there exists a non-zero solution $(b_0,b_1,\dots, b_\mu)\in k^{\mu+1}$ of $\begin{pmatrix} b_0 & b_1 &  \cdots & b_\mu\end{pmatrix}M_{N}(x_0)=0$. Set $\varPsi(x)=\sum_{i=0}^{\mu}b_ix^i$. Then $r_{\varPsi,n}(x_0)=0$ for all $n\in \{\nu, \dots, N-1\}$ and using the previous theorem, we see that $\xi \in \tilde{J}[N]$.
\end{proof}

\begin{cor}\label{cor:maxrank}
The matrix $M_{N}(x)$ with polynomial coefficients has maximal rank $\mu+1$.
\end{cor}

\begin{proof}
If not, $M_{N}(x_0)$ would not have maximal rank for all $x_0\in k$. Since $k$ is infinite and $J[N]$ is a finite set, this contradicts the Theorem.
\end{proof}

Let $\tilde{U}^0_{N}(x)$ denote the gcd of the polynomials $\varPi_{j,N}(x)$ as $j$ runs over $\cS_0$.

\begin{cor} \label{cor:UfN} If $N\geq 2g+1$, then $\tilde{U}_{N}(x)$ is equal to the radical of the prime-to-$F(x)$ part of $\tilde{U}^0_{N}(x)$.
\end{cor}

\begin{proof}
By Corollary~\ref{cor:maxrank},  $\det{\varSigma_j(x)}\neq 0$ for at least one $j\in \cS_0$. The result therefore follows from Theorem~\ref{th:nscJN} and the fact that $\varPi_{j,N}(x)=\det{\varSigma_j(x)}/F(x)^{\mu(\mu+1)/2}$.   
\end{proof}

\begin{theorem}\label{th:genbound} Let $j\in \cS_0$.

\case{a} Let $x_0\in k$ with $F(x_0)\neq 0$. If $(x_0,y_0)\in \tilde{J}[N]$, then $\varPi_{j,N}$ has a zero of order at least $N-j_{\mu+1}$ at $x_0$.

\case{b} Suppose $\varPi_{j,N}(x)\neq 0$. Then $(N-j_{\mu+1})\sharp\, (X\cap \tilde{J}[N]) \leq  4g\sum_{\ell=1}^{\mu+1}j_\ell-2g\mu(\mu+1)$.
\end{theorem} 

\begin{proof}  
\case{a} Let $\alpha$, $\varPsi(x)$ be as in the proof of Theorem~\ref{th:mainth}. Recall that $\varPsi\neq 0$, so denote its degree by $\delta$ and write $\varPsi(x) =b_0+b_1x+\cdots  +b_\delta x^\delta$, $b_\delta\neq 0$.  Let $L_i$ denote the $i^{\text{th}}$ row of $\varSigma_j(x)$ and $\varSigma'_j(x)$ denote the matrix obtained from $\varSigma_j(x)$ by replacing row $\delta+1$ by $L'_\delta:=b_0L_0+b_1L_1+\cdots   +b_\delta L_\delta$. Then  $\det{\varSigma_j(x)}=b_\delta \det{\varSigma'_j(x)}$. But the entry of $L'_\delta$ in the $\ell^{\text{th}}$ column is
\begin{equation*}
b_0s_{0,j_\ell}+b_1s_{1,j_\ell}+\cdots + b_\delta s_{\delta,j_\ell}=(D_{j_\ell}\Psi(x)y)(2y+Q(x))^{2j_\ell-1}=(D_{j_\ell}\alpha)(2y+Q(x))^{2j_\ell-1}
\end{equation*}
since $j_\ell \geq \nu$. This has a zero of order at least $N-j_\ell$ at $x_0$. Expanding $\det{\varSigma'_j(x)}$ by row $\delta+1$ then shows that $\det{\varSigma'_j(x)}$ has a zero of order at least $\min_{1\leq \ell\leq \mu+1}{\{N-j_\ell\}}=N-j_{\mu+1}$ at $x_0$. Hence the same is true of $\det{\varSigma_j(x)}$.

\case{b} This follows from \case{a} and Corollary~\ref{cor:degPij}, on recalling that each $x_0\in \tilde{\cT}_N$ corresponds to two points $(x_0,y_0)$, $(x_0,-y_0-Q(x_0))\in \tilde{J}[N]$. 
\end{proof}

Corollary~\ref{cor:maxrank} and Theorem~\ref{th:genbound} should be viewed as the central results of this paper, since they imply a general bound for $\sharp\, (X\cap \tilde{J}[N])$.   We summarize this in the following Corollary.

\begin{cor} \label{cor:bound}
With notation as in Theorem~\ref{th:genbound} define $\cS:=\{j\in \cS_0\mid  \varPi_{j,N}(x)\neq 0\}$.  Then $\cS$ is non empty, and
\begin{equation*}
\sharp\, (X\cap \tilde{J}[N])\leq \min_{(j_1,\dots,  j_{\mu+1})\in \cS}\frac{4g\sum_{\ell=1}^{\mu+1}j_\ell-2g\mu(\mu+1)}{N-j_{\mu+1}}. 
\end{equation*}
\end{cor}

\begin{proof}
The fact that $\cS$ is non-empty follows from Corollary~\ref{cor:maxrank} and the bound then follows from \case{b} of Theorem~\ref{th:genbound}. 
\end{proof}

\subsection{Best and worst cases, proofs of Theorem~\ref{th:helltorsionbound} and Proposition~\ref{prop:2g12}} \label{subsec:bestworst}The optimal situation occurs when $j=(\nu,\nu+1,\dots, \nu+\mu)$ belongs to  $\cS$, which leads to the leftmost subdeterminant of $M_N(x)$. Recall that by definition $\varDelta_{N}(x)=\varPi_{N}(x)$ for this choice of $j$.  

\begin{proof}[Proof of Theorem~\ref{th:helltorsionbound}]
When $j=(\nu,\nu+1,\dots, \nu+\mu)$, we have $j_{\mu+1}=N-g$, so $N-j_{\mu+1}=g$.  By Theorem~\ref{th:genbound}, $\tilde{U}_N(x)^g$ divides $\Delta_N(x)$ so \case{a} of Theorem~\ref{th:helltorsionbound} follows from Corollary~\ref{cor:degDelta}.  \case{b} follows at once from \case{a}.  To prove \case{c}, note first that by Lemma~\ref{lem:not} every term in the products appearing in the formula for the leading coefficient of $\varDelta_N(x)$ (see Corollary~\ref{cor:degDelta}) has absolute value at most $N$. It follows that this coefficient is divisible only by primes $p\leq N$. \end{proof}

The worst possible bound comes from the rightmost subdeterminant, which corresponds to the choice $j=(N-\mu-1, N-\mu,\dots,  N-1)$. Indeed, this choice of $j$ maximizes the numerator of the bound in Corollary~\ref{cor:bound} and the denominator is $N-j_{\mu+1}=1$.  It therefore results in a bound that applies even if $j\notin \cS$. Simplifying the numerator then gives the following Proposition.  

\begin{prop} \label{prop:badgenbound}
For all $g$, $N\geq 2g+1$ and for all hyperelliptic $X$ of genus $g$ embedded in their jacobian using a Weierstrass point as base point, we have
\begin{equation*}
\sharp\, (X\cap \tilde{J}[N]])\leq \begin{cases} g(N^2-(2g)^2) &\text{if $N$ is even}\\
g(N^2-(2g-1)^2) &\text{if $N$ is odd}.
\end{cases}
\end{equation*}
\end{prop}

The two bounds of Proposition~\ref{prop:2g12} are the cases $N=2g+1$ and $N=2g+2$ of Proposition~\ref{prop:badgenbound}.  Note that, for fixed $g$ this is weaker than Proposition~\ref{prop: generalbound} except when $N$ is close to $2g+1$ or when $p$ is purely inseparable for $X$ and $N-1$ is a sufficiently large power of $p$.

\section{Proof of Proposition~\ref{prop:divCantor}} \label{sec:Cantorprop} We recall the statement for the reader's convenience.

\begin{propnn}Let $N\geq 2g+1$ and let $r\in \{0,1,\dots ,2g-2\}$. Define 
\begin{equation*}
\ve_{r,g}=\left(g-\left\lfloor{\frac{r+1}{2}}\right\rfloor \right)\left(\left\lfloor{\frac{r}{2}}\right\rfloor+1    \right).
\end{equation*}
Then $\tilde{U}_{N}(x)^{\ve_{r,g}}$ divides $\varDelta_{N+r}(x)$. 
\end{propnn}
The case $r=0$ of this Proposition is simply \case{a} of Theorem~\ref{th:genbound} applied with $j=(\nu,\nu+1,\dots, \nu+\mu)$. Our proof of 
Proposition~\ref{prop:divCantor} is a generalization of the proof of this. Indeed, we could also generalize Theorem~\ref{th:genbound} to subdeterminants of $M_{N+r}(x)$ other than $\varGamma_{N}(x)$. Note however that when $r\geq 1$, then in view of Theorem~\ref{th:nscJN},  $\tilde{U}_N(x)$  cannot divide \emph{all} the subdeterminants of $M_{N+r}(x)$.

We exploit the following simple Lemma.

\begin{lemma} \label{lem:recMfN}  Let $N\geq 2g+1$ be an integer.

\case{a} If $N$ is odd,  the matrix $M_{N+1}(x)$ is obtained from $M_{N}(x)$ by suppressing the leftmost column and adding the column which is the transpose of $\begin{pmatrix} s_{0,N} & s_{1,N} & \dots & s_{\mu,N}  \end{pmatrix}$ on the right.

\case{b} If $N$ is even, then $M_{N+1}(x)$ is obtained from $M_{N}(x)$ by adding the row $\begin{pmatrix} s_{\mu+1,\nu}& s_{\mu+1,\nu+1}&\cdots  & s_{\mu+1, N}\end{pmatrix} $ at the bottom and the column defined as the transpose of $\begin{pmatrix} s_{0,N}& s_{1,N}& \cdots & s_{\mu+1,N}\end{pmatrix}$, on the right.

\case{c} In both cases, $M_{N+2}(x)$ is obtained from $M_{N}(x)$ by suppressing the leftmost column and adding the row $\begin{pmatrix} s_{\mu+1,\nu}& s_{\mu+1,\nu+1}&\cdots  &  s_{\mu+1,N+1}\end{pmatrix} $ at the bottom and two columns, the transpose of $\begin{pmatrix} s_{0,N}& s_{1,N}& \cdots & s_{\mu+1,N}\end{pmatrix}$ followed the transpose of $\begin{pmatrix} s_{0,N+1} & s_{1,N+1}& \cdots & s_{\mu+1,N+1}\end{pmatrix}$, on the right.
\end{lemma}

\begin{proof}
Note that, if $N$ is increased to $N+1$, $\mu$ is unchanged if $N$ is odd and increases by $1$ if $N$ is even, while $\nu$ is unchanged if $N$ is even and increases by one if $N$ is odd.  \case{a} and \case{b} are then obvious from the definition of $M_{N}(x)$. \case{c} follows from \case{a} and \case{b}. 
\end{proof}

\begin{proof}[Proof of Proposition~\ref{prop:divCantor}]
It suffices to show that $\tilde{U}_{N}(x)^{\ve_{r,g}}$ divides $\varGamma_{N+r}(x)$.  Let $x_0\in \tilde{\cT}_{N}$ and let $\alpha=\varPhi(x)+\varPsi(x)y$ have divisor $N[(x_0,y_0)]-N[\infty]$.    We consider several cases according to the parity of $N$ and $r$. 

\case{a} Suppose $N$ odd and $r$ even, $r=2q$ with $0\leq q\leq g-1$. Then  $\nu_{N+r}=\nu+q$ and $\mu_{N+r}=\mu+q$ so by Lemma~\ref{lem:not} $\mu_{N+r}+\nu_{N+r}=N-g+r$. Thus $\varGamma_{N+r}(x)$ is the determinant of the matrix
\begin{equation*}  M_{P,Q,N+r}^0(x)=
\begin{pmatrix}
s_{0,\nu+q} & s_{0,\nu+q+1} &    \cdots   &   s_{0,N-g+r} \\
s_{1,\nu+q} & s_{1,\nu+q+1} &    \cdots   &   s_{1,N-g+r} \\
\vdots         & \vdots              &    \ddots   &     \vdots        \\
s_{\mu+q,\nu+q} & s_{\mu+q,\nu+q+1} &  \cdots    & s_{\mu+q, N-g+r}
\end{pmatrix}.
\end{equation*}

Let $0\leq \ell\leq q$.  Since $N$ is odd, $\varPsi$ has degree $\mu$ and can be assumed monic.  Then $(x-x_0)^\ell\varPsi(x)$ has degree $\mu+\ell$ and, since $(x-x_0)^\ell\varPhi(x)$ has degree at most $\nu+\ell-1$ and $(x-x_0)^\ell \alpha$ has a zero of order $N+\ell$ at $x_0$,  $D_n((x-x_0)^\ell\alpha)=D_n((x-x_0)^\ell\varPsi(x)y)$ whenever $n\geq \nu+\ell$ and, if $\nu+\ell\leq n\leq N+q$, $D_n((x-x_0)^\ell\varPsi(x)y)$ has a zero of order at least $N+\ell-n$ at $x_0$.   Writing
\begin{equation*}
(x-x_0)^\ell\varPsi(x)=c_0^{(\ell)}+c_1^{(\ell)}x+\cdots + c_{\mu+\ell}^{(\ell)}x^{\mu+\ell}, \qquad c_{\mu+\ell}^{(\ell)}=1,
\end{equation*}
we find 
\begin{align*}
D_n(((x-&x_0)^\ell\varPsi(x)y)\\ 
&=\big(c_0^{(\ell)}s_{0,n}+c_1^{(\ell)}s_{1,n}+\cdots   + c_{\mu+\ell-1}^{(\ell)}s_{\mu+\ell-1,n}+s_{\mu+\ell,n}\big)(2y+Q(x))^{1-2n}
\end{align*}
for all $n$ such that $\nu+\ell\leq n\leq N+\ell$. Let $L_0$, $L_1$, \dots, $L_{\mu+q}$ be the rows of $M_{N+r}^0(x)$ from the top down. By the above, replacing $L_{\mu+\ell}$ by $c_0^{(\ell)}L_0+c_1^{(\ell)}L_1+\cdots  +c_{\mu+\ell-1}^{(\ell)}L_{\mu+\ell-1}+L_{\mu+\ell}$ does not change the determinant. Furthermore, starting from the left, the polynomials in this row have zeros of order at least $(N+\ell)-(\nu+q)=g+\mu+(\ell-q)$, $(N+\ell)-(\nu+q+1)=g+\mu+(\ell-q-1)$, \dots, the lower bound decreasing by one each time until either the rightmost column is attained or the row ends in a sequence of polynomials that in general do not vanish at $x_0$. Denoting by $\cP_d$ a polynomial with a zero of order at least $d$ at $x_0$ and  by a star a polynomial which in general does not vanish at $x_0$ we find that $\varGamma_{N+r}(x)$ is equal to the determinant of a matrix of the form
\begin{equation*} \tag*{$(*)$}
\begin{pmatrix}
s_{0,\nu+q} & s_{0,\nu+q+1} &    \cdots   &  s_{0,N-g+r-1} & s_{0,N-g+r} \\
s_{1,\nu+q} & s_{1,\nu+q+1} &    \cdots    &  s_{1,N-g+r-1} & s_{1,N-g+r} \\
\vdots         & \vdots              &    \ddots    &  \vdots   & \vdots \\
s_{\mu-1,\nu+q} & s_{\mu-1,\nu+q+1} &     \cdots &  s_{\mu-1,N-g+r-1} & s_{\mu-1,N-g+r} \\
\cP_{g+\mu-q}    &  \cP_{g+\mu-q-1}    &  \cdots       & \cP_{g-2q+1} &   \cP_{g-2q} \\    
\cP_{g+\mu-q+1}& \cP_{g+\mu-q}        &  \cdots    &   \cP_{g-2q+2} &    \cP_{g-2q+1} \\
  \vdots         & \vdots              &    \ddots    &  \vdots  & \vdots  \\
  \cP_{g+\mu-1}      &   \cP_{g+\mu-2}&  \cdots    &   \cP_{g-q} & \cP_{g-q-1}  \\ 
 \cP_{g+\mu}      &   \cP_{g+\mu-1}&  \cdots    &   \cP_{g-q+1} & \cP_{g-q}
\end{pmatrix}
\end{equation*}
when $q<\frac{g}{2}$ and of the form
\begin{equation*}  \tag*{$(**)$}  \hskip-13mm
\begin{pmatrix}
s_{0,\nu+q} & s_{0,\nu+q+1} &    \cdots   &   s_{0,\nu+r} &   s_{0,\nu+r+1} &  \cdots &  s_{0,N-g+r-1} & s_{0,N-g+r} \\
s_{1,\nu+q} & s_{1,\nu+q+1} &    \cdots   &   s_{1,\nu+r} &   s_{1,\nu+r+1} &  \cdots &  s_{1,N-g+r-1} & s_{1,N-g+r} \\
\vdots         & \vdots              &    \ddots   &     \vdots       & \vdots    &              \ddots  &  \vdots   & \vdots \\
s_{\mu-1,\nu+q} & s_{\mu-1,\nu+q+1} &    \cdots   &   s_{\mu-1,\nu+r} &   s_{\mu-1,\nu+r+1} &  \cdots &  s_{\mu-1,N-g+r-1} & s_{\mu-1,N-g+r} \\
\cP_{g+\mu-q}    &  \cP_{g+\mu-q-1}    &  \cdots    &   \cP_1       &   *   &  \cdots    & * &   * \\    
\cP_{g+\mu-q+1}& \cP_{g+\mu-q}        &  \cdots    &   \cP_2       &  \cP_1 &  \cdots  &  *&    * \\
 \vdots         & \vdots              &    \ddots   &     \vdots       & \vdots    &              \ddots  &  \vdots  & \vdots  \\
 \cP_{g+\mu-(g-q)}   &  \cP_{g+\mu-(g-q+1)}  &    \cdots   &     \cP_{\mu}         &    \cP_{\mu-1}     &            \cdots    & \cP_1 &   *     \\
\cP_{g+\mu-(g-q-1)}   &  \cP_{g+\mu-(g-q)}  &    \cdots   &     \cP_{\mu+1}         &    \cP_{\mu}     &            \cdots    & \cP_2 &   \cP_1     \\ 
  \vdots         & \vdots              &    \ddots   &     \vdots       & \vdots    &              \ddots  &  \vdots  & \vdots  \\
 \cP_{g+\mu-1}      &   \cP_{g+\mu-2}&  \cdots  &   \cP_{g+\mu-q-1} & \cP_{g+\mu-q-2} & \cdots  &   \cP_{g-q} & \cP_{g-q-1}  \\ 
 \cP_{g+\mu}      &   \cP_{g+\mu-1}&  \cdots  &   \cP_{g+\mu-q} & \cP_{g+\mu-q-1} & \cdots  &   \cP_{g-q+1} & \cP_{g-q}
\end{pmatrix}. 
\end{equation*}
when $q\geq \frac{g}{2}$. In both cases there are $q+1$ rows of polynomials $\cP_d$.

In the expansion of the determinant, the products that contribute the zeros of smallest order are those that contain subproducts that appear in the expansion of the determinant of the $(q+1)\times (q+1)$  block at the bottom right of the matrices, namely
\begin{equation*}
\begin{pmatrix} \cP_{g-q} & \cP_{g-q-1} & \cdots     & \cP_{g-2q} \\
\cP_{g-q+1} & \cP_{g-q}  &   \cdots     &   \cP_{g-2q+1}  \\
\vdots          &  \vdots      &   \ddots     &    \vdots            \\
\cP_{g}        &   \cP_{g-1} &  \cdots      &     \cP_{g-q} 
\end{pmatrix},
\end{equation*} 
when the large matrix has the form $(*)$ where some of the $\cP_{i}$'s in the upper right part of the matrix have to be replaced by stars when it has the form $(**)$. These subproducts all have zeros of order at least $(g-q)(q+1)$. Since $r=2q$, $\lfloor{\frac{r}{2}}\rfloor=\lfloor{\frac{r+1}{2}}\rfloor=q$, we have $(g-q)(q+1)=\ve_{r,g}$ and the result is proved in this case. 

\case{b} Suppose that $N$ is odd and $r$ is odd. Write $r=2q+1$, so that $0\leq q\leq g-2$.  Then the matrix $M_{N+r}^0(x)$ is obtained from  $M_{N+2q}^0(x)$ by suppressing the column on the left and adding a column on the right. Arguing as before, this means that $M_{N+r}^0(x)$ has the same determinant as a matrix of the form 
\begin{equation*}
\begin{pmatrix}
s_{0,\nu+q+1} & s_{0,\nu+q+2} &    \cdots   &  s_{0,N-g+r-1} & s_{0,N-g+r} \\
s_{1,\nu+q+1} & s_{1,\nu+q+2} &    \cdots    &  s_{1,N-g+r-1} & s_{1,N-g+r} \\
\vdots         & \vdots              &    \ddots    &  \vdots   & \vdots \\
s_{\mu-1,\nu+q+1} & s_{\mu-1,\nu+q+2} &     \cdots &  s_{\mu-1,N-g+r-1} & s_{\mu-1,N-g+r} \\
\cP_{g+\mu-q-1}    &  \cP_{g+\mu-q-2}    &  \cdots       & \cP_{g-2q} &   \cP_{g-2q-1} \\    
\cP_{g+\mu-q}& \cP_{g+\mu-q-1}        &  \cdots    &   \cP_{g-2q+1} &    \cP_{g-2q} \\
  \vdots         & \vdots              &    \ddots    &  \vdots  & \vdots  \\
  \cP_{g+\mu-2}      &   \cP_{g+\mu-3}&  \cdots    &   \cP_{g-q-1} & \cP_{g-q-2}  \\ 
 \cP_{g+\mu-1}      &   \cP_{g+\mu-2}&  \cdots    &   \cP_{g-q} & \cP_{g-q-1}
\end{pmatrix}
\end{equation*}
when $q<(g-1)/2$ and a matrix of a similar form to $(**)$ when $q\geq (g-1)/2$. Again there are $q+1$ rows of $\cP_d$'s and arguing as in case \case{a} we find that all terms in the expansion of the determinant have a zero of order at least $(g-q-1)(q+1)$. Since $r=2q+1$, $q+1=\lfloor{\frac{r+1}{2}}\rfloor=\lfloor{\frac{r}{2}}\rfloor+1$ so $(g-q-1)(q+1)=\ve_{r,g}$ and we again obtain the desired result. 

\case{c} Suppose that $N$ is even and $r$ is even and again write $r=2q$ with $0\leq q\leq g-1$. Then $\varPhi$ has degree $\nu-1$ . Let $\ell\in \{0,1,\dots , q\}$. Then $(x-x_0)^{\ell}\alpha$  has a zero of order $N+\ell$ at $x_0$, so again  $D_n(x-x_0)^{\ell}\alpha$ has a zero of order at least $N+\ell-n$ provided $n\leq N-\ell$. If $n\geq \nu+\ell$, then $D_n((x-x_0)^{\ell}\alpha)=D_n((x-x_0)^\ell\varPsi(x)y)$.  Also, $\varPsi$ has degree at most $\mu$. But by Lemma~\ref{lem:recMfN} it is non-zero so we may assume it is monic. Denote its degree by $\delta$, so $0\leq \delta\leq \mu$. Then 
\begin{equation*}
(x-x_0)^\ell\varPsi(x)=c_{0}^{(\ell)}+c_{1}^{(\ell)}x+\cdots + c_{\delta+\ell}^{(\ell)}x^{\delta+\ell}, \qquad c_{\delta+\ell}^{(\ell)}=1
\end{equation*}
so that
\begin{align*}
D_n((x-&x_0)^\ell\varPsi(x)y)\\
&=\big(c_{0}^{(\ell)}s_{0,n}+c_{1}^{(\ell)}s_{1,n}+\cdots + c_{\delta+\ell-1}^{(\ell)}s_{\delta+\ell-1,n}+s_{\delta+\ell,n}\big)(2y+Q(x))^{1-2n}
\end{align*}
for all $n$ such that $\nu+\ell\leq n\leq N+\ell$. 
We can thus apply a similar argument to that in case \case{a}.  This gives rise to a matrix of the form
\begin{equation*}
\begin{pmatrix}
s_{0,\nu+q} & s_{0,\nu+q+1} &    \cdots   &  s_{0,N-g+r-1} & s_{0,N-g+r} \\
s_{1,\nu+q} & s_{1,\nu+q+1} &    \cdots    &  s_{1,N-g+r-1} & s_{1,N-g+r} \\
\vdots         & \vdots              &    \ddots    &  \vdots   & \vdots \\
s_{\delta-1,\nu+q} & s_{\delta-1,\nu+q+1} &     \cdots &  s_{\delta-1,N-g+r-1} & s_{\delta-1,N-g+r} \\
\cP_{g+\mu-q}    &  \cP_{g+\mu-q-1}    &  \cdots       & \cP_{g-2q+1} &   \cP_{g-2q} \\    
\cP_{g+\mu-q+1}& \cP_{g+\mu-q}        &  \cdots    &   \cP_{g-2q+2} &    \cP_{g-2q+1} \\
  \vdots         & \vdots              &    \ddots    &  \vdots  & \vdots  \\
 \cP_{g+\mu-1}      &   \cP_{g+\mu-2}&  \cdots    &   \cP_{g-q} & \cP_{g-q-1}  \\ 
\cP_{g+\mu}      &  \cP_{g+\mu-1}&  \cdots    &   \cP_{g-q+1} & \cP_{g-q} \\
s_{\delta+1+q, \nu+q}&  s_{\delta+1+q, \nu+q+1} & \cdots &   s_{\delta+1+q, N-g+r-1} & s_{\delta+1+q, N-g+r} \\
\vdots         & \vdots              &    \ddots    &  \vdots  & \vdots  \\
s_{\mu+q, \nu+q}&  s_{\mu+q, \nu+q+1} & \cdots  & s_{\mu+q, N-g+r-1} & s_{\mu+q, N-g+r}
\end{pmatrix}
\end{equation*}
when $q\leq g/2$ and a matrix of a form similar to $(**)$ when $q>g/2$.  (Here the top block of $s_{i,n}$'s is empty if $\delta=0$ and the bottom block of $s_{i,n}$'s is empty if $\delta=\mu$.) 

Arguing with the block of $P_d$'s  as in \case{a} we find again there is a zero of order at least $(g-q)(q+1)=\ve_{r,g}$.

\case{d} Suppose that $N$ is even and $r$ is odd, $r=2q+1$ where $0\leq q\leq g-2$. This time $M_{N+r}^0(x)$ is obtained from $M_{N+2q}^0(x)$ by adding a row at the bottom and a column on the right.  Because of the extra row, we have to let $\ell$ vary from $0$ to $q+1$. Arguing as in case \case{c}, we see that $\varGamma_{N+r}$ is equal to the determinant of a matrix of the form
\begin{equation*}
\begin{pmatrix}
s_{0,\nu+q} & s_{0,\nu+q+1} &    \cdots   &  s_{0,N-g+r-1} & s_{0,N-g+r} \\
s_{1,\nu+q} & s_{1,\nu+q+1} &    \cdots    &  s_{1,N-g+r-1} & s_{1,N-g+r} \\
\vdots         & \vdots              &    \ddots    &  \vdots   & \vdots \\
s_{\delta-1,\nu+q} & s_{\delta-1,\nu+q+1} &     \cdots &  s_{\delta-1,N-g+r-1} & s_{\delta-1,N-g+r} \\
\cP_{g+\mu-q}    &  \cP_{g+\mu-q-1}    &  \cdots       & \cP_{g-2q} &   \cP_{g-2q-1} \\    
\cP_{g+\mu-q+1}& \cP_{g+\mu-q}        &  \cdots    &   \cP_{g-2q+1} &    \cP_{g-2q} \\
  \vdots         & \vdots              &    \ddots    &  \vdots  & \vdots  \\
 \cP_{g+\mu}      &   \cP_{g+\mu-1}&  \cdots    &   \cP_{g-q} & \cP_{g-q-1}  \\
 *     &   \cP_{g+\mu}&  \cdots    &   \cP_{g-q+1} & \cP_{g-q}  \\
 s_{\delta+1+q, \nu+q}&  s_{\delta+1+q, \nu+q+1} & \cdots &   s_{\delta+1+q, N-g+r-1} & s_{\delta+1+q, N-g+r} \\
\vdots         & \vdots              &    \ddots    &  \vdots  & \vdots  \\
s_{\mu+q, \nu+q}&  s_{\mu+q, \nu+q+1} & \cdots  & s_{\mu+q, N-g+r-1} & s_{\mu+q, N-g+r}
\end{pmatrix}
\end{equation*}
when $q<\frac{g+1}{2}$ and a matrix resembling $(**)$ when $q\geq \frac{g+1}{2}$.  Note the star in the bottom left corner of the block of $\cP_d$'s, which occurs because $(x-x_0)^{q+1}\varPhi(x)$ has degree $\nu+q$, so $D_{\nu+q}((x-x_0)^{q+1}\varPhi(x))\neq 0$. There are $q+2$ rows of $\cP_d$'s and one checks that when the determinant is expanded, the products that contribute the zeros of smallest order are those that contain subproducts occuring in the expansion of the $(q+1)\times (q+1)$ block whose entries on the principal diagonal are all $\cP_{g-q-1}$'s.  This again leads to an order of vanishing of at least $(g-q-1)(q+1)=\ve_{r,g}$. 
\end{proof}
 
\section{Examples}\label{sec:examples}  It is easy to program the preceding results on a computer, using Lemma~\ref{lem:Cmn} to compute the $s_{i,n}$'s and the matrix $M_N(x)$, or Corollary~\ref{cor:pidet} to compute the $\varPi_{j,N}$'s.  One can then deduce $\tilde{U}_{N}(x)$ from Corollary~\ref{cor:UfN}.  In this section we give a number of examples illustrating this. 

\begin{ex:cantor}  \label{ex:cantor} We consider the curve $y^2+y=x^5$, so that $g=2$, $P(x)=x^5$, $Q(x)=1$ and $F(x)=4x^5+1$ (see also \cite{Ca94}, page 143). We suppose $p\neq 5$ in order that $X_{P,Q}$ be smooth. Choose $N=5$. Then $\nu=3$ and $\mu=0$, so $\varDelta_{5}(x)=s_{3}(x)$ and
\begin{equation*}
M_{5}(x)=\begin{pmatrix} s_{3}(x) & s_{4}(x) \end{pmatrix} =\begin{pmatrix}  10x^2(x^5-1)^2 & -5x(x^5-1)(x^{10}-27x^5+1) \end{pmatrix}.
\end{equation*} 

Theorem~\ref{th:nscJN} implies that $x_0\in \tilde{\cT}_5$ if and only if $M_5(x_0)$ is the zero matrix.  If $p\neq 2$, this happens if and only if $x_0(x_0^5-1)=0$.  Therefore $\tilde{U}_{5}(x)=x(x^5-1)$ and $\sharp\, (X\cap \tilde{J}[5])=12$, as was already observed in \cite{BoGrLe01}. 

If $p=2$, then $\varDelta_5(x)=0$ and $s_4(x)$ simplifies to $x^{16}-x$. This polynomial has simple roots and is coprime to $F(x)$. We deduce from Theorem~\ref{th:nscJN} that $\tilde{U}_{5}(x)=x^{16}-x$.  Hence $\sharp\, (X\cap \tilde{J}[5])=32$, as was already observed in \cite{BoGrLe01}. 

Observe that $\tilde{U}_{5}(x)^2$ divides $\varDelta_5(x)$ as predicted by Proposition~\ref{prop:divCantor}.   
\end{ex:cantor}

\begin{ex:sixtorsion} \label{ex:sixtorsion} Next consider the curve $X=X_{P,0}$ where $P(x)=x^5-x$. Then $X$ is smooth if $p=0$ or $p\geq 5$.  Again, $g=2$. Let $N=6$.  We have $\nu=4$ and $\mu=0$, so $M_{6}(x)=\begin{pmatrix} s_{4}(x) & s_{5}(x) \end{pmatrix}$ and $\Delta_{6}(x)=\Gamma_{6}(x)=s_{4}(x)$. Writing  $\theta(x)=\left( {{x}^{4}}-2 {{x}^{3}}+2 {{x}^{2}}+2 x+1\right) \left( {{x}^{4}}+2 {{x}^{3}}+2 {{x}^{2}}-2 x+1\right)$, we find
\begin{equation*}
s_{4}(x)=-5\theta(x)^2, \qquad s_{5}(x)=  2{\theta(x) (3 {{x}^{12}}+291 {{x}^{8}}+161 {{x}^{4}}-7) }.
\end{equation*}

Suppose that $p\neq 5$. Then Theorem~\ref{th:nscJN} says that $(x_0,y_0)\in X\cap \tilde{J}[6]$ if and only if $\theta(x_0)=0$. Therefore, $\tilde{U}_{6}(x)=\theta(x)$. Again, $\theta(x)^2$ divides the $s_{4}(x)$ as predicted by Proposition~\ref{prop:divCantor}. We have $\sharp\, (X\cap \tilde{J}[6])=16$.

Suppose that $p=5$. Then $s_4(x)$ vanishes, so Theorem~\ref{th:nscJN} says that $(x_0,y_0)\in X \cap \tilde{J}[6]$ if and only if $x_0$ is a root of $s_5(x)$ but not of $x^5-x$. Hence $\tilde{U}_{6}(x)$ is equal to the radical of the prime-to-$(x^5-x)$ part of $s_5(x)$.  One can check that $s_5(x)$ has no repeated roots and is coprime to $x^5-x$; since $s_5(x)$ has degree $20$, $\sharp\, (X\cap \tilde{J}[6])=40$.
\end{ex:sixtorsion}

\begin{ex:sixtorsionfive}  The above example suggests that, when $g=2$, we look at points of order $6$ in characteristic $5$ in more detail. Let $P(x)$ be monic of degree $5$ and without repeated roots and write $P(x)=x^5+ax^4+bx^3+cx^2+dx+e$ with $a$, $b$, $c$, $d$, $e\in k$.  For the curve $X_{P(x),0}$, we find 
\begin{align*}
\varDelta_{6}(x)=\varGamma_{6}(x)=s_{4}(x)&=-2a(x^{15}+ax^{14}+bx^{13}+cx^{12}+dx^{11})+\\
&(-a e+( b-2 {{a}^{2}})  d-2 {{c}^{2}}+a b c+{{b}^{3}})x^{10}+\lambda_9(x),
\end{align*}
where $\lambda_9(x)$ has degree at most $9$. It follows that if $a\neq 0$, then $\Gamma_{6}(x)$ has degree $15$;  by Theorem~\ref{th:genbound}~\case{a}, $\varDelta_{6}(x)$ has a multiple root at any $x_0$ with $(x_0,y_0)\in X\cap \tilde{J}[6]$. It follows that $\sharp\, (X\cap \tilde{J}[6])\leq 2\lfloor{\frac{15}{2}}\rfloor=14$.  On the other hand, when $a=0$, we see that $\varDelta_{6}(x)$ has degree at most $10$, so if $\varDelta_{6}(x)$ does not vanish, $\sharp\, (X\cap \tilde{J}[6])\leq 10$.

When $\varDelta_{6}(x)$ vanishes (as happens in Example~\ref{ex:sixtorsion}), we have to use the right hand entry $s_{5}(x)$ of $M_{6}(x)$ to bound $\sharp\, (X\cap \tilde{J}[6])$.  A computation shows that this happens if and only if $a=b=c=0$, which since $p=5$ is the case if and only if $X$ is isomorphic to $X_{x^5-x,0}$. Hence $\sharp\, (X\cap \tilde{J}[6])=40$. 

This example shows that, for fixed $g$, $N$ and $p$, we can sometimes obtain bounds that depend on $P$ and $Q$.
\end{ex:sixtorsionfive}

\begin{ex:seventorsion}  We return to the curve $X:y^2+y=x^5$ and we want to determine whether there are any characteristics $p\neq 5$ for which $X\cap J[7]^*$ is non-empty. We have 
\begin{equation*}
M_{7}(x)=\begin{pmatrix} s_{04}(x) & s_{05}(x) & s_{06}(x) \\ s_{14}(x) & s_{15}(x) & s_{16}(x)  \end{pmatrix},
\end{equation*}
(where $s_{i,n}$ is as in paragraph~\ref{sec:hyperprelims}) and the three polynomials $\det{\Sigma_j(x)}$ are $m_{12}(x)=s_{04}(x)s_{15}(x)-s_{05}(x)s_{14}(x)$, $m_{13}(x)=s_{04}(x)s_{16}(x)-s_{06}(x)s_{14}(x)$ and $m_{23}(x)=s_{05}(x)s_{16}(x)-s_{15}(x)s_{06}(x)$.  In particular, $\Gamma_{7}=m_{12}$ and $\Delta_{7}=\Gamma_{7}/F$. 

Recall that $F(x)=Q(x)^2+4P(x)=4x^5+1$. We find  $m_{12}(x)=-5\tilde{U}_5(x)^2F(x)(7 {{x}^{20}}-1218 {{x}^{15}}-463 {{x}^{10}}-198 {{x}^{5}}-3)$, $m_{13}(x)=5\tilde{U}_5(x)F(x)(14 {{x}^{30}}-6594 {{x}^{25}}+16110 {{x}^{20}}+2970 {{x}^{15}}+3285 {{x}^{10}}-159 {{x}^{5}}-1 )$ and $m_{23}(x)=-F(x)(14 {{x}^{40}}-11172 {{x}^{35}}+28112 {{x}^{30}}-295344 {{x}^{25}}+1330 {{x}^{20}}-111384 {{x}^{15}}-1598 {{x}^{10}}-582 {{x}^{5}}-1)$.  The fact that $\tilde{U}_5^2$ divides $m_{12}$ illustrates Proposition~\ref{prop:divCantor}.   By Theorem~\ref{th:nscJN}, $(x_0,y_0)$ belongs to $X\cap J[7]^*$ if and only if $x_0$ is a common root of the factors of degree $20$, $30$ and $40$ of $m_{12}(x)$, $m_{13}(x)$ and $m_{23}(x)$. A computation involving resultants shows that they have a common root if and only if $p=911$, in which case $X\cap J[7]^*$ consists of the points $(x_0,y_0)$ such that $x_0^5=433$.  
\end{ex:seventorsion}

\begin{ex:sevengthree} We consider the curve $X$ with affine model $y^2+y=x^7$, so $X$ has genus $3$ when $p\neq 7$. We find $M_{7}(x)=\begin{pmatrix} s_{4}(x) & s_{5}(x) & s_{6}(x) \end{pmatrix}$, where $s_4(x)=7 {{x}^{3}} \left( 5 {{x}^{21}}+58 {{x}^{14}}-73 {{x}^{7}}+5\right)$, $s_5(x)=-7 {{x}^{2}}\left( 2 {{x}^{28}}+324 {{x}^{21}}-1044 {{x}^{14}}+232 {{x}^{7}}-3\right)$ and $s_6(x)= x\left( 2 {{x}^{35}}+1826 {{x}^{28}}-12030 {{x}^{21}}+6264 {{x}^{14}}-407 {{x}^{7}}+1\right)$. Here $x$ is a common factor of $s_{4}(x)$, $s_{5}(x)$ and $s_{6}(x)$, so $(0,0)$ and $(0,-1)$ are points of order $7$.   A computation using resultants shows that $X\cap J[7]^*=\{(0,0),(0,-1)\}$ in every characteristic $p \neq 7$. 
\end{ex:sevengthree}

\begin{ex:x015w3} As a final example, we apply our results to a modular curve.  Here genus two curves have been studied by many authors (see for example \cite{Po01}), so we move on to genus three. For convenience, we seek an example with a rational Weierstrass point. According to Galbraith \cite{Ga96}, page~46, the modular curve $X:=X_0(51)/W_3$ (where $W_3$ is the Atkin-Lehner involution associated to the divisor $3$ of $51$) has genus $3$ and an affine equation $s^2=(t^4-2t^3+3t^2-6t+5)(t^3-5t^2+3t-3)(t+1)$. Thus $(-1,0)$ is a Weierstrass point and the substitution $s=204y/x^4$, $t+1=-204/x^3$ sends it to infinity and gives the equation ${{y}^{2}}={{x}^{7}}+536 {{x}^{6}}+136272 {{x}^{5}}+21016080 {{x}^{4}}+2122416000 {{x}^{3}}+136819425024 {{x}^{2}}+4946281998336 x+72074394832896$, which is of the form to which  the methods of this paper can be applied.  We try to compute $\tilde{U}_N(x)$ successively for $N=7$, $8$, \dots{} using Corollaries~\ref{cor:pidet} and  \ref{cor:UfN}.   This soon starts to become slow as $N$ increases. We can use a number of tricks to accelerate the computations. Firstly, if $p$ is a prime of good reduction not dividing $N$, then the reduction map on $J[N]$ is injective. Adding the subscript $p$ to denote taking the fiber at $p$, we deduce that $X\cap \tilde{J}[N]$ injects into $X_p\cap \tilde{J}_p[N]$, so for example if $\tilde{U}_{N,X_p}(x)$ is a non-zero constant, then  $\tilde{U}_{N,X}(x)$ is also a non-zero constant and therefore $X\cap \tilde{J}[N]$ is empty. Our model has good reduction away from $2$, $3$ and $17$, and working at small $p$ of good reduction gives considerable speed-up. Secondly, in practice, it is often not necessary to compute the gcd of all the $\varPi_{j,N}$'s, but only of a small subset. Finally, if $x-x_0$ seems to be a factor of $\tilde{U}_N(x)$, one can compute $N(x_0,y_0)$ in $J$ to check whether $(x_0,y_0)$ indeed has order dividing $N$.    After some computation, we found that $X\cap \tilde{J}[N]$ is empty for all $N$ such that $7\leq N\leq 34$, $N\neq 32$, and that $X\cap {J}[32]^*$ consists of the two points $(0, \pm 8489664)$.   

\end{ex:x015w3}

We end the section with a family of examples of curves $X$ such that $\sharp\, (X\cap J[N])\geq N^2$ for infinitely many $N$. 

\begin{prop} Let $\ell$ be an odd prime such that $p$ is a generator of the group $(\bZ/\ell\bZ)^\times$ and let $X$ be the hyperelliptic curve with model $y^2+y=x^\ell$.  Then there are infinitely many integers $N$ for which $\sharp\, (X\cap J[N])\geq N^2$. 
\end{prop}

\begin{proof}
We think of $X$ as defined over the field $\bF_p$ with $p$ elements, and we use some well-known properties of curves and their jacobians over finite fields that go back to Weil. The genus of $X$ is $g=\frac{\ell-1}{2}$. For every integer $r\geq 1$, let $\chi_r$ denote the characteristic polynomial of the Frobenius endomorphism $\pi^r$ of $J$ over the extension $\bF_{p^r}$ of degree $r$ of $\bF_p$. We begin by proving that $\chi_1(t)=t^{2g}+p^g$. This is probably well-known, but for the convenience of the reader we sketch the proof. We first compute the number of points of $X$ over $\bF_{p^r}$ for $r\in \{1,2,\dots, g\}$.  Fix such an $r$. By the hypotheses on $p$ and $\ell$, $\bF_{p^r}$ contains no $\ell^{\text{th}}$ root of unity other than $1$. It follows that the map $x\mapsto x^\ell$ is a permutation of $\bF_{p^r}$. Suppose $p\neq 2$. When $x^\ell=-\frac{1}{4}$, there is a unique $y$ such that $y^2+y=x^\ell$. Together with $\infty$, this contributes $2$ points to $X(\bF_{p^r})$. There are $\frac{p^r-1}{2}$ values of $x$ such that $y^2+y=x^\ell$ has two solutions in $\bF_{p^r}$ and a further $\frac{p^r-1}{2}$ values of $x$ such that $y^2+y=x^\ell$ has no solutions in $\bF_{p^r}$.  The first possibility contributes a further $p^r-1$ points. We conclude that $\sharp\, X(\bF_{p^r})=p^r+1$.  A similar argument when $p=2$ gives the same result.

It is well-known that, for a general curve of genus $g$, one has  $\sharp\, X(\bF_{p^r})=p^r+1-\tr{\pi^r}$, where $\tr$ is the trace map of the action on the $q$-adic Tate module for any prime $q\neq p$.  Our calculation therefore shows that $\tr{\pi^r}=0$ for all $r\in \{1,2,\dots,  g\}$. An argument involving symmetric functions of the roots now shows that the coefficient of  degree $i$ vanishes for all $i\in \{g,g+1,\dots, 2g-1\}$. The functional equation $t^{2g}\chi_1(\frac{p}{t})=p^g\chi_1(t)$ then implies that the coefficients of degree $i\in \{1,2,\dots, g-1\}$ also vanish and that $\chi_1(t)=t^{2g}+p^g$ as claimed.  By factoring $\chi_1(t)$ as a complex polynomial, we find that its roots are $\zeta\sqrt{p}$ where $\zeta$ runs over the set of complex numbers with $\zeta^{2g}=-1$. 

From now on suppose that $r\equiv 2g \pmod{4g}$. Since $\zeta^r=-1$ for all $\zeta$ as above and the roots of $\chi_r(t)$ are the $r^{\text{th}}$ powers of those of $\chi_1(t)$, we deduce that $\chi_r(t)=(t+p^{r/2})^{2g}$.  The  semisimplicity of the action of $\pi^r$ on Tate modules now implies that $\pi^r$ is multiplication by $-p^{r/2}$. Let  $N=N_r=p^{r/2}+1$. It is well-known that  $\sharp\,J(\bF_{p^r})=\chi_r(1)$, so that $\sharp\, J(\bF_{p^r})=N^{2g}$.  Also, $J(\bF_{p^r})$ is killed by $\pi^{r/2}-1=-N$ and hence it is killed by $N$; since $\sharp \, J[N]=N^{2g}$ elementary abelian group theory implies that $J(\bF_{p^r})=J[N]$.  It follows that $X\cap J[N]=X(\bF_{p^r})$. On the other hand, $\sharp\, X(\bF_{p^r})=p^r+1-\tr{\pi^r}$, which gives $\sharp\, X(\bF_{p^r})=p^r+1+2gp^{r/2}$.  Since $N=p^{r/2}+1$, we conclude that $\sharp\, (X\cap J[N])=(p^{r/2}+1)^2+2(g-1)p^{r/2}=N^2+2(g-1)(N-1)$. Hence $\sharp\, (X\cap J[N])\geq N^2$ and the result follows.
\end{proof}

\section{Appendix: proof of Proposition~\ref{prop: generalbound}}\label{sec:proofgeneralbound}

Let $k$, $p$, $X$, $J$, $g$ be as in the Introduction. We suppose $g\geq 2$.    We apply the arguments of \S~4 of  \cite{CoKaRi99}; however these authors work in characteristic zero so we briefly review the proof. For any integer $d\geq 1$, write $X^{(d)}$ for the $d^{\mathrm{th}}$ symmetric power of $X$, in other words the quotient of $X^d$ by the group of permutation of the coordinates. Thus $X^{(d)}$ parametrizes effective divisors of degree $d$ on $X$. The embedding $X\to J$ induces morphisms $u_d:X^{(d)}\to J$ for all $d$ and it is known that $u_g$ is birational and surjective. Write $W_d=u_d(X^{(d)})$; in particular $W_{g-1}$ is a theta divisor also denoted $\Theta$. Since $u_1$ is an embedding, we identify $X$ with $W_1$. 

We use some facts and terminology from \cite{ACGH85} (see in particular pages 7--13)\footnote{These authors use complex analytic methods, but the results we need are of a purely algebraic nature.}. Let $D$ be an effective divisor of degree $d$, where $1\leq d\leq g$. Write $\cL(D)$ for the vector subspace of $k(X)$ consisting of rational functions whose polar divisor is $\leq D$. Let $\varphi:X\to \bP^{g-1}$ be the canonical morphism. Then $\varphi(X)$ does not lie in any proper projective subspace of $\bP^{g-1}$, and a special case of a geometric form of the Riemann-Roch theorem asserts that 
\begin{equation*}
\dim_k{\cL(D)}=d-\dim{(\overline{\varphi(D)})},
\end{equation*} 
where $\overline{\varphi(D)}$ is the intersection of all hyperplanes $H\subseteq \bP^{g-1}$ such that $D\leq \varphi^*(H)$ and $\dim{}$ is the projective dimension. Recall that $D$ is \emph{special} if $\dim_k(\cL(K-D))\geq 1$, where $K$ is a canonical divisor. There is a notion of \emph{exceptional special divisor}; if $d\leq g$ then $D$ is exceptional special if and only if $\dim{(\overline{\varphi(D)})}<d-1$. An exceptional special divisor is special, and the converse is true if $d\leq g$. 

The following two lemmas are classical, but most modern texts mention at best generic versions, with points belonging to some unspecified dense open subsets. This would be insufficient for our purposes.

\begin{lemma}\label{lem:notexspec}
Suppose $1\leq d\leq g$ and $D$ is not exceptional special. Then there exists a divisor $E$ of degree $g-d$ such that $D+E$ is not special. Furthermore, if $S\subseteq X$ is finite, we can choose $E$ in such a way that its support is disjoint from $S$.
\end{lemma}

\begin{proof} Since $D$ is not exceptional special, $\dim{(\overline{\varphi(D)})}=d-1$ and we can choose $E=\sum_{j=1}^{g-d}[z_j]$, where the $z_j$'s are such that $\overline{\varphi(D)}$ and $\varphi(z_j)$ span $\bP^{g-1}$. Then $D+E$ is not exceptional special hence not special (since it has degree $g$). Since the set of special divisors is closed,  the points $z_j$ can be chosen outside any finite set $S$. 
\end{proof}

If $P\in J$, denote by $t_P:J\to J$ the translation-by-$P$ morphism $t_P(Q)=P+Q$. 

\begin{lemma}\label{lem:goodQexists}
Let $Q\in J$ with $Q\notin W_1$. Then there exists $P\in J$ such that $W_1\subseteq t_P\Theta$ but $Q\notin t_P\Theta$.
\end{lemma}

\begin{proof}  The conclusion is equivalent to $t_{-P}W_1\subseteq \Theta$ but $t_{-P}Q\notin \Theta$. This is what we shall prove, replacing $P$ by $-P$. 

So let $d$ be such that $Q\in W_d-W_{d-1}$.  By hypothesis $d\geq 2$. Choose $y_i$ ($1\leq i \leq d$) such that $Q=u_d(D)$, where $D=\sum_{i=1}^d [y_i]$. By hypothesis, no $y_i$ is $\infty$.  Also, $D$ is not exceptional special, since if it were, it would be special and therefore linearly equivalent to an effective divisor whose support includes $\infty$, contrary to the minimality of $d$. By  Lemma~\ref{lem:notexspec}, we can find $E$ of degree $g-d$ such that $D+E$ is not special. Furthermore, we can suppose that the support of $E$ does not contain $\infty$. This implies that $D+E\notin X^{(g-1)}$. Put $P=u_{g-d}(E)$. Then $P+Q\notin W_{g-1}=\Theta$. Also, since $d\geq 2$, $g-d\leq g-2$, so $P\in W_{g-2}$ and $t_PW_1\subseteq W_{g-1}$. 
\end{proof}

If $m\in \bZ$ and $m\neq 0$, write $mX$ for the push-forward $m_*X$ of $X$ under the multiplication-by-$m$ isogeny $m:J\to J$.

\begin{lemma}\label{lem:CKRlemma}
Let $m$ be an integer such that $|m|\geq 2$. If either $p$ is not purely inseparable for $X$ or $m$ is not a power of $p$, then $X\cap mX$ is finite of cardinality  at most $gm^2$. 
\end{lemma}

\begin{proof}  We first show that $X\cap mX$ is finite. Since $X$ is irreducible and reduced of dimension one, the same is true for $mX$. Hence it suffices to show that $mX\neq X$.  We show that in fact $mX$ is not isomorphic to $X$. This is clear if $mX$ is not smooth, so suppose that $mX$ is smooth. Factor the finite map $m:X\to mX$ as $X\xrightarrow{\iota}Z\xrightarrow{\sigma} mX$, where $\sigma$ is separable and $\iota$ is purely inseparable. Then $Z$ is isomorphic to some Frobenius twist $X^{(p^r)}$ of $X$ (see for example \cite{Liu02}, Proposition~7.4.21) and therefore has genus $g$.  Suppose $mX$ is isomorphic to $X$. Then $Z$ and $mX$ have the same genus $g\geq 2$, and it follows from the Hurwitz genus formula that $\sigma$ has degree $1$, and is therefore an isomorphism. Let $\Jac(Z)$ denote the Jacobian variety of $Z$ and view $Z$ as embedded in $\Jac(Z)$ using $\iota(\infty)$ as base point. Then $m:J\to J$ is the composition of the purely inseparable isogeny $J\to \Jac(Z)$ induced by $\iota$ and the isomorphism $\Jac(Z) \to J$ induced by $\sigma$.  But the hypotheses on $m$ implies that the separable part of $m:J\to J$ is not an isomorphism which is a contradiction. 

To bound the cardinality we apply Lemma~\ref{lem:goodQexists} and an intersection theory argument. It is known that $\Theta$ is symmetric and ample, so that $m^*\Theta$ is algebraically equivalent to $m^2\,\Theta$. Also, $(X.\Theta)=g$. There exists $Q\in mX$ such that $Q\notin X$. Hence there exists $P\in J$ such that $X\subseteq t_P\Theta$, but  $Q\notin t_P\Theta$. Since $mX$ is irreducible, this implies that $mX\cap t_P\Theta$ is finite and 
\begin{equation*}
\sharp\,(mX\cap X) \leq \sharp\,(mX\cap t_P\Theta)\leq (mX.t_P\Theta)= (mX.\Theta)
=(X.m^*\Theta)=m^2(X.\Theta)=gm^2
\end{equation*}
as required.
\end{proof}

\begin{proof}[Proof of Proposition~\ref{prop: generalbound}]

\case{a} Suppose $X\cap J[2]^*$ non-empty and let $\xi$ be an element. Then $2[\xi]-2[\infty]$ is a principal divisor, so $\cL(2[\infty])$ contains a non-constant function and so $2[\infty]$ is a special divisor, which implies that $X$ is hyperelliptic and $\infty$ is a Weierstrass point. The converse follows from Lemma~\ref{lem:first} since $F(x)$ has degree $2g+1$ when $p=0$ or $p\geq 2$ and degree at most $g$ when $p=2$.

\case{b} If $m\equiv 1\pmod{N}$, then multiplication by $m$ acts trivially on $J[N]$. It follows that $X\cap J[N]$ is contained in $X\cap mX$.  We apply  Lemma~\ref{lem:CKRlemma} with $m$ chosen with $|m|$ as small as possible subject to the condition $mX\neq X$. If $p$ is not purely inseparable for $X$ or $N-1$ is not a power of $p$, this gives $m=1-N$; otherwise we find that $m=1+N$ except when $N=3$ and $p=2$, in which case $m=-5$.  
\end{proof}

\end{document}